\documentclass[a4paper]{article}
\usepackage{etex}
\usepackage{pgfplots}

\usepackage{enumerate}
\usepackage{amsmath}

\usepackage{amsthm}

\usepackage{amssymb}
\usepackage{latexsym}
\usepackage{float}
\usepackage{amssymb}
\usepackage{latexsym}
\usepackage{tikz}
\usepackage{tikz-cd}
\usepackage{braids}
\usetikzlibrary{
  knots,
  hobby,
  decorations.pathreplacing,
  shapes.geometric,
  calc
}

\newtheorem{Theorem}{Theorem}
\newtheorem{Proposition}{Proposition}

\theoremstyle{definition}
\newtheorem{definition}{Definition}

\def\S{{\mathbf S}}

\begin{document}

\title{Cyclic branched covers of alternating knots}

\author{
Luisa Paoluzzi
}
\date{\today}
\maketitle

\begin{abstract}

\vskip 2mm

For any integer $n>2$, the $n$-fold cyclic branched cover $M$ of an alternating
prime knot $K$ in the $3$-sphere determines $K$, meaning that if $K'$ is a knot 
in the $3$-sphere that is not equivalent to $K$ then its $n$-fold cyclic 
branched cover cannot be homeomorphic to $M$.

\vskip 2mm


\vskip 2mm


\end{abstract}

\section{Introduction}

A knot $K$ in the $3$-sphere is \emph{alternating} if it admits a generic 
projection onto a $2$-sphere where the double points of the projection 
alternate between overcrossings and undercrossings when travelling along the 
knot. In spite of the purely combinatorial character of this definition, being 
alternating seems to have deep consequences on the topological and geometric 
properties of the knot. For instance, alternating prime knots cannot be 
satellite knots \cite{Me}, that is their exteriors are atoroidal. Nonetheless, 
until fairly recently no description of this class of knots in terms of 
geometric or topological propertiess was known. A characterisation of 
alternating knots as the class of knots admitting spanning surfaces with 
special features was provided independently by Greene \cite{G2} and Howie 
\cite{H}. 

In this work, we are interested in studying another topological aspect of
alternating knots, namely the behaviour of their cyclic brached covers. Recall
that given a knot $K$ in the $3$-sphere and an integer $n\ge 2$ one can
construct a closed $3$-manifold $M(K,n)$ called (the total space of) the 
\emph{$n$-fold cyclic branched cover of $K$}. We refer the reader to Rolfsen's
book \cite{R} for the explicit construction of these manifolds since it will 
not be needed here (see also \cite{P2} for a survey on cyclic branched covers). 
In the following we will only need the following fact: there is a map 
$p:M(K,n)\longrightarrow \S^3$ whose restriction 
$M(K,n)\setminus p^{-1}(K)\longrightarrow \S^3\setminus K$ is a covering map. 
We point out that the manifolds $M(K,n)$ can be considered as topological 
invariants of the knot $K$. It was shown by Kojima in \cite{K} that its 
$n$-fold cyclic branched cover $M(K,n)$ \emph{determines} the prime knot $K$ 
provided that $n$ is sufficiently large in the following sense: for each pair 
of prime knots $K$ and $K'$, there exists an integer $N=N(K,K')$ such that if 
the manifolds $M(K,n)$ and $M(K',n)$ are homeomorphic for some $n\ge N$, then 
$K$ and $K'$ are necessarily equivalent. 

If the $n$-fold cyclic branched cover of a knot $K$ does not determine $K$,
that is if there exists another knot $K'$ not equivalent to $K$ such that 
$M(K,n)$ and $M(K',n)$ are homeomorphic, then we say that $K$ and $K'$ are
\emph{$n$-twins}.

Kojima's result can thus be restated by saying that a prime knot does not have
$n$-twins if $n$ is large enough. However, it should be stressed that for every
integer $n$ there are prime knots, and even hyperbolic knots, that have
$n$-twins (see, for instance, \cite{N,S,Z}). 

The main result of the paper asserts that cyclic branched covers are strong
invariants for alternating prime knots.

\begin{Theorem}\label{t:main}
Let $K$ be an alternating prime knot. If $n>2$, then $K$ has no $n$-twins.
\end{Theorem}

Before introducing the key ideas of the proof, it is worth mentioning what
happens in the situations where the hypothesis of the theorem are not
fulfilled. First of all, this type of result cannot hold for composite knots.
Indeed, as shown by Viro \cite{V}, it is easy to construct non equivalent 
composite knots, and even alternating ones, that are $n$-twins for every 
$n\ge 2$. This is a consequence of the fact that the construction of cyclic 
branched covers does not depend on a chosen orientation of the knot. An example 
of such twins is provided by the connected sum of the non invertible knot 
$8_{17}$ (in Rolfsen's notation) with itself and the connected sum of the same 
knot with its reverse (for the non invertibility of $8_{17}$ see \cite{KL}). 

As cyclic branched covers of composite knots are not prime while the $3$-sphere
is irreducible, two $n$-twins are either both composite or both prime, so the 
only situation left to consider is that of $2$-twins that are prime. Here we 
are somehow disregarding the case of the trivial knot which is neither 
composite nor prime (as the unit of the monoid structure induced by composition
on the set of oriented knots). One can easily see that, for all $n\ge 2$, the 
$n$-fold cyclic branched cover of the trivial knot is the $3$-sphere. It 
follows then from the positive answer to Smith's conjecture \cite{MB} that, for 
every $n\ge 2$, the trivial knot has no $n$-twins. 

Turning our attention to the case of $2$-twins of prime knots, we see that
Montesinos knots \cite{Mo} admit $2$-twins as soon as their Montesinos 
presentation consists of at least four tangles and at least three distinct 
tangles appear (this latter condition is not necessary, though). Since 
alternating Montesinos knots with these properties clearly exist, the 
conclusion of Theorem~\ref{t:main} is false if $n=2$. 

Notice that the given examples of $2$-twins are instances of a more general 
phenomenon called \emph{Conway mutation}: if $S$ is a $2$-sphere that meets a
knot $K$ in four points, one can remove one of the two $3$-balls bounded by $S$
and the tangle it contains and glue it back in a different way, obtaining thus
a new knot $K'$ (see for instance \cite[p. 259]{P2} for more details on this 
construction). The new knot $K'$ is obtained from $K$ by \emph{Conway mutation} 
and is a \emph{Conway mutant} of $K$. It turns out that $K$ and $K'$ have 
homeomorphic $2$-fold branched covers. They need not be $2$-twins, for they may 
be equivalent. For the two knots to be $2$-twins, it is necessary (but possibly 
not sufficient) that the sphere $S$ is an \emph{essential Conway sphere}, that 
its intersection with the exterior of $K$ is incompressible and 
$\partial$-incompressible in the knot exterior. Possibly the most famous Conway 
mutants that are also $2$-twins but not Montesinos knots (nor alternating) are 
the Conway \cite{C} and Kinoshita-Terasaka \cite{KT} knots. 

It was shown by Greene in \cite{G1} that two prime alternating $2$-twins are 
necessarily Conway mutants. Greene conjectures that if a prime alternating knot 
$K$ admits a $2$-twin $K'$ then $K'$ is necessarily a Conway mutant of $K$ and, 
in particular, is itself alternating.

This conjecture is all the more striking if one considers that there is a 
plethora of phenomena giving rise to $2$-twins of a prime knot, as discussed
for instance in \cite{P2}. Among the different constructions of $2$-twin knots, 
there is one, originally introduced by Nakanishi \cite{N} and Sakuma \cite{S}, 
that can be exploited to produce $n$-twins for any $n\ge 2$. Given a 
two-component link $L=L_1\cup L_2$ in the $3$-sphere both of whose components 
are trivial knots, one can consider the $n$-fold cyclic branched cover of the 
component $L_i$, $i=1,2$. Since $L_i$ is the trivial knot, the manifold 
$M(L_i,n)$ is the $3$-sphere. For $j\neq i$, the preimage of $L_j$ in 
$M(L_i,n)$ is connected, and thus a knot $K_j$, provided that $n$ and the 
linking number of $L_1$ and $L_2$ are coprime. If this is the case, by 
construction the knots $K_1$ and $K_2$ have homeomorphic $n$-fold cyclic 
branched covers. Intuitively, if no homeomorphism of $\S^3$ exchanges the two 
components of $L$ one can expect $K_1$ and $K_2$ to be non equivalent, that is 
they are genuine $n$-twins. 

The most remarkable fact about this construction was pointed out by Zimmermann
in \cite{Z}. He proved that, if $K$ is a hyperbolic knot admitting an $n$-twin 
$K'$ for some $n>2$, then $K$ and $K'$ must be obtained as the $K_1$ and $K_2$ 
of Nakanishi and Sakuma's construction just described. This is one of the key 
points in the proof of the main result. Indeed, prime alternating knots are 
either hyperbolic or torus knots. It is probably folklore that the conclusion 
of Theorem~\ref{t:main} holds for torus knots of any kind. We will provide a
proof for the sake of completeness in Proposition~\ref{p:torus} of 
Section~\ref{s:proof}. As a consequence, one is only left to consider the 
hyperbolic case. 

The second key ingredient in the proof of the main result is that certain types
of symmetries of prime alternating knots are \emph{visible} on an alternating
diagram. The notion of visibility of symmetries on minimal diagrams of 
prime alternating knots has been considered as a straightforward consequence of 
the proof of Tait's flyping conjecture (see, for instance, \cite{HTW}) by 
Menasco and Thistlethwaite \cite{MT}. Unfortunately, most of the time no 
precise definition of the meaning of ``visible" seem to be provided in the 
literature where the notion appears. 

In our case, we need certain symmetries, i.e. \emph{periods} (see 
Section~\ref{s:quotients} for the definition), to be visible in a very specific
sense (see Section~\ref{s:quotients} for the actual definition of visibility we
need). A detailed proof of this fact was given recently by Costa and Quach
Hongler in \cite{CQ} (see also \cite{Bo} for the case of symmetries of odd 
prime order). The importance of this fact as a central tool in the proof 
becomes clear once we restate the aforementioned result of Zimmermann's in the 
following way. Let $K$ be a hyperbolic knot and $n>2$. $K$ admits an $n$-twin 
if and only if $K$ admits an $n$-period $\psi$ such that the quotient knot
$K/\langle \psi \rangle$ is the trivial knot and no homemorphism of $\S^3$
exchanges the components of the link $(K\cup Fix(\psi))/\langle \psi \rangle$.
Incidentally, we note that prime satellite $n$-twins, with $n>2$, need not be
related in this way according to the examples given in \cite{BPa}.

\bigskip

The organisation of the paper is the following. In Section~\ref{s:quotients} we
show that if an alternating knot $K$ admits a period $\psi$ which is visible on
a minimal alternating diagram and such that $K/\langle \psi \rangle$ is the 
trivial knot, then the quotient admits a diagram of a specific form. This
uses properties of alternating diagrams of the trivial knot. In 
Section~\ref{s:symmetry}, we exploit the structure of the link 
$(K\cup Fix(\psi))/\langle \psi \rangle$ determined in the previous section to
show that the there is a homemorphism of $\S^3$ exchanging its components.
Finally, in Section~\ref{s:proof} we prove Theorem~\ref{t:main}.


\section{Quotients diagrams of prime alternating knots via special periodic
symmetries}\label{s:quotients}

In this section we determine a diagram of the quotient of an alternating knot
$K$ by the action of a period under the hypotheses that the period is visible 
on a minimal alternating diagram for $K$ and the quotient knot is trivial.

We need the following definitions.

\begin{definition}
A \emph{period of order $n$} or \emph{$n$-period} of a knot $K$ is 
an orientation preserving diffeomorphism $\psi$ of order $n\ge 2$ of the 
$3$-sphere which leaves $K$ invariant and such that its fixed-point set is a 
circle disjoint from $K$.
\end{definition}

\begin{definition}
Let $K$ be a knot admitting a period $\psi$ of order $n$. We say that
\emph{$\psi$ is visible on a diagram $D$ for $K$} if there exists a $2$-sphere
$S$ embedded in $\S^3$ and a projection $p:\S^3\setminus\{ * \}\longrightarrow 
S$ such that $p(K)=D$, $\psi(S)=S$, and there is a diffeomorphism 
$\hat\psi: S\longrightarrow S$ of order $n$ such that $\hat\psi\circ
p=p\circ\psi$. 
\end{definition}

\begin{figure}[ht]
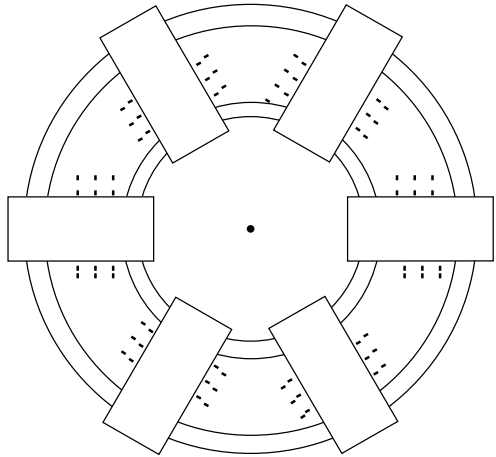

 \begin{center}



\ifx\XFigwidth\undefined\dimen1=0pt\else\dimen1\XFigwidth\fi
\divide\dimen1 by 2896
\ifx\XFigheight\undefined\dimen3=0pt\else\dimen3\XFigheight\fi
\divide\dimen3 by 2707
\ifdim\dimen1=0pt\ifdim\dimen3=0pt\dimen1=4143sp\dimen3\dimen1
  \else\dimen1\dimen3\fi\else\ifdim\dimen3=0pt\dimen3\dimen1\fi\fi
\tikzpicture[x=+\dimen1, y=+\dimen3]
{\ifx\XFigu\undefined\catcode`\@11
\def\temp{\alloc@1\dimen\dimendef\insc@unt}\temp\XFigu\catcode`\@12\fi}
\XFigu4143sp
\ifdim\XFigu<0pt\XFigu-\XFigu\fi
\clip(213,-3148) rectangle (3109,-441);
\tikzset{inner sep=+0pt, outer sep=+0pt}
\pgfsetlinewidth{+7.5\XFigu}
\draw  (1665,-1800) circle [radius=+765];
\draw  (1665,-1800) circle [radius=+1222];
\filldraw  (1661,-1790) circle [radius=+19];
\draw  (1661,-1790) circle [radius=+670];
\pgfsetdash{}{+0pt}
\draw  (1661,-1790) circle [radius=+1340];
\pgfsetfillcolor{white}
\filldraw (225,-1599)--(225,-1982)--(1087,-1982)--(1087,-1599)--cycle;
\filldraw (1100,-459)--(769,-650)--(1200,-1397)--(1532,-1205)--cycle;
\filldraw (1549,-2390)--(1218,-2198)--(787,-2945)--(1118,-3136)--cycle;
\filldraw (2559,-644)--(2227,-453)--(1797,-1200)--(2128,-1391)--cycle;
\pgfsetlinewidth{+30\XFigu}
\pgfsetdash{{+15\XFigu}{+90\XFigu}}{+15\XFigu}
\draw (540,-1575)--(900,-1575);
\draw (540,-2025)--(900,-2025);
\draw (1305,-720)--(1485,-1035);
\draw (2025,-720)--(1845,-1035);
\draw (900,-2655)--(1080,-2385);
\draw (1305,-2880)--(1485,-2565);
\draw (1845,-2610)--(2025,-2925);
\draw (2430,-1575)--(2790,-1575);
\draw (2475,-2025)--(2835,-2025);
\draw (2520,-990)--(2295,-1305);
\pgfsetlinewidth{+7.5\XFigu}
\pgfsetdash{}{+0pt}
\filldraw (2235,-1599)--(2235,-1982)--(3097,-1981)--(3097,-1599)--cycle;
\pgfsetlinewidth{+30\XFigu}
\pgfsetdash{{+15\XFigu}{+90\XFigu}}{+15\XFigu}
\draw (2250,-2385)--(2430,-2700);
\pgfsetlinewidth{+7.5\XFigu}
\pgfsetdash{}{+0pt}
\filldraw (2101,-2193)--(1769,-2384)--(2201,-3130)--(2532,-2939)--cycle;
\pgfsetlinewidth{+30\XFigu}
\pgfsetdash{{+15\XFigu}{+90\XFigu}}{+15\XFigu}
\draw (1980,-675)--(1755,-1035);
\draw (2475,-945)--(2250,-1260);
\draw (2430,-1485)--(2835,-1485);
\draw (2475,-2070)--(2835,-2070);
\draw (2295,-2340)--(2475,-2655);
\draw (1800,-2655)--(1980,-2925);
\draw (1350,-2925)--(1530,-2610);
\draw (855,-2610)--(1035,-2340);
\draw (540,-2070)--(900,-2070);
\draw (540,-1485)--(900,-1485);
\draw (1350,-675)--(1530,-990);
\draw (900,-945)--(1080,-1260);
\draw (855,-990)--(1035,-1305);
\endtikzpicture%

\end{center}
  \caption{A schematic diagram where a $6$-period is visible.}
 \label{fig:symm}
\end{figure}

Figure~\ref{fig:symm} shows a schematic diagram where a $6$-period is 
visible. The dot in the middle is one of the two intersections of the 
fixed-point set of the period with the sphere of projection. The crossings of 
the knot can be arranged to sit inside six equal tangles, represented by
rectangles in the figure, that are permuted by the period which acts on the 
sphere as a rotation about the central dot (and the point at infinity).

Notice that according to the given definition of visibility a symmetry of the 
knot that acts freely cannot be visible if it has order $>2$ for it cannot 
leave invariant a $2$-sphere.

For simplicity, in the following we will abuse notation and write $\psi$ 
even when referring to $\hat\psi$.

\begin{Proposition}\label{p:quodiag}
Let $K$ be a non trivial prime alternating knot with period $\psi$ of oder $n$. 
Assume that $\psi$ is visible on a minimal diagram $D$ for $K$ and that 
$K/\langle \psi \rangle$ is the trivial knot. Then, up to diagram isotopy
relative to the axis, the diagram $D/\langle \psi \rangle$ is alternating and 
of the form shown on the left-hand side of Figure~\ref{fig:quodiag}, where the 
dot represents the axis of $\psi$ and each box a sequence of crossings.
\end{Proposition}

\begin{figure}[ht]
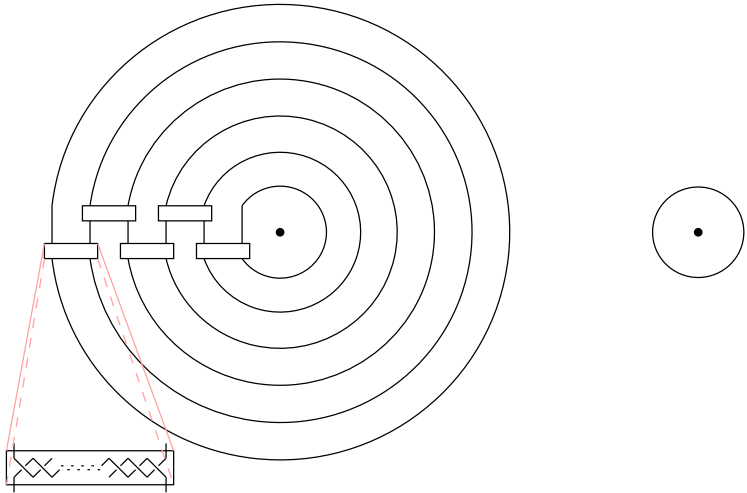

 \begin{center}



\ifx\XFigwidth\undefined\dimen1=0pt\else\dimen1\XFigwidth\fi
\divide\dimen1 by 4385
\ifx\XFigheight\undefined\dimen3=0pt\else\dimen3\XFigheight\fi
\divide\dimen3 by 2933
\ifdim\dimen1=0pt\ifdim\dimen3=0pt\dimen1=4143sp\dimen3\dimen1
  \else\dimen1\dimen3\fi\else\ifdim\dimen3=0pt\dimen3\dimen1\fi\fi
\tikzpicture[x=+\dimen1, y=+\dimen3]
{\ifx\XFigu\undefined\catcode`\@11
\def\temp{\alloc@1\dimen\dimendef\insc@unt}\temp\XFigu\catcode`\@12\fi}
\XFigu4143sp
\ifdim\XFigu<0pt\XFigu-\XFigu\fi
\definecolor{pink2}{rgb}{1,0.63,0.63}
\clip(123,-3117) rectangle (4508,-184);
\tikzset{inner sep=+0pt, outer sep=+0pt}
\pgfsetlinewidth{+7.5\XFigu}
\draw (180,-2902)--(293,-3015);
\draw (293,-2902)--(405,-3015);
\draw (742,-2902)--(855,-3015);
\draw (855,-2902)--(968,-3015);
\draw (968,-2902)--(1080,-3015);
\draw (293,-2902)--(247,-2947);
\draw (225,-2970)--(180,-3015);
\draw (338,-2970)--(293,-3015);
\draw (405,-2902)--(360,-2947);
\draw (450,-2970)--(405,-3015);
\draw (855,-2902)--(810,-2947);
\draw (968,-2902)--(923,-2947);
\draw (1080,-2902)--(1035,-2947);
\draw (1012,-2970)--(968,-3015);
\draw (900,-2970)--(855,-3015);
\draw (787,-2970)--(742,-3015);
\draw (742,-2902)--(698,-2947);
\draw (180,-2813)--(180,-2902);
\draw (1080,-2813)--(1080,-2902);
\draw (1080,-3015)--(1080,-3105);
\draw (180,-3015)--(180,-3105);
\draw (135,-2857) rectangle (1125,-3060);
\pgfsetdash{{+15\XFigu}{+45\XFigu}}{+15\XFigu}
\draw (698,-2947)--(450,-2947);
\draw (450,-2970)--(698,-2970);
\pgfsetdash{}{+0pt}
\draw (1530,-1395) arc[start angle=+145.0, end angle=+-145.0, radius=+274.6];
\draw (1305,-1395) arc[start angle=+160.7, end angle=+-160.7, radius=+476.8];
\draw (1080,-1395) arc[start angle=+166.87, end angle=+-166.87, radius=+693.1];
\draw (855,-1395) arc[start angle=+170.07, end angle=+-170.07, radius=+913.7];
\draw (630,-1395) arc[start angle=+172.03, end angle=+-172.03, radius=+1136];
\pgfsetdash{}{+0pt}
\draw (405,-1395) arc[start angle=+173.35, end angle=+-173.35, radius=+1359.2];
\pgfsetdash{}{+0pt}
\filldraw  (4230,-1553) circle [radius=+22];
\draw  (4230,-1553) circle [radius=+270];
\filldraw  (1755,-1553) circle [radius=+22];
\draw (1575,-1710) rectangle (1260,-1620);
\draw (1350,-1485) rectangle (1035,-1395);
\draw (1125,-1710) rectangle (810,-1620);
\draw (900,-1485) rectangle (585,-1395);
\draw (675,-1710) rectangle (360,-1620);
\draw (1305,-1485)--(1305,-1620);
\draw (1080,-1485)--(1080,-1620);
\draw (855,-1485)--(855,-1620);
\draw (630,-1485)--(630,-1620);
\draw (405,-1395)--(405,-1620);
\draw (1530,-1395)--(1530,-1620);
\pgfsetstrokecolor{pink2}
\pgfsetdash{{+60\XFigu}{+60\XFigu}}{++0pt}
\draw (360,-1710)--(135,-3060);
\draw (675,-1710)--(1125,-3060);
\pgfsetdash{}{+0pt}
\draw (360,-1620)--(135,-2857);
\draw (675,-1620)--(1125,-2857);
\endtikzpicture%

\end{center}
  \caption{The structure of the diagram $D/\langle \psi \rangle$ of the
quotient knot $K/\langle \psi \rangle$ for a tangle of size $k=5$ on the left,
and the case $k=1$ on the right. Each rectangular box represents a sequence of
crossings.}
 \label{fig:quodiag}
\end{figure}

\begin{proof}
Before proving the proposition, let us explain the structure of the diagram on
the left-hand side of Figure~\ref{fig:quodiag}. The central dot represents one 
of the two intersections of $Fix(\psi)/\langle \psi \rangle$ with the sphere of 
projection, the other being the point at infinity. Each box represents a 
horizontal (with respect to the picture) sequence of crossings, as suggested by 
the drawing below the diagram. Remark that because the diagram is alternating 
the sign of crossings in one box constrains the signs of the crossings in all 
other boxes. This means that each diagram is determined, up to taking a mirror, 
simply by the \emph{size $k$ of the tangle}, i.e. number of arcs, and the 
number of crossings in each of the $k-1$ boxes. Observe also that if the roles 
of the two intersection points of $Fix(\psi)/\langle \psi \rangle$ with the 
sphere of projection are exchanged, then the diagram changes (the size is the 
same, while the numbers of crossings in the boxes appear in reversed order) but 
its structure stays the same. Indeed, the boxes in the picture are alternately 
located below and above with respect to each other, but this relative position 
can be changed since the boxes can be slid around the dot by an isotopy of the 
$2$-sphere leaving the intersections with $Fix(\psi)/\langle \psi \rangle$ 
fixed.

\begin{figure}[ht]
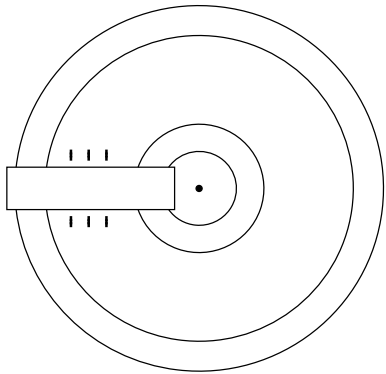

 \begin{center}



\ifx\XFigwidth\undefined\dimen1=0pt\else\dimen1\XFigwidth\fi
\divide\dimen1 by 2250
\ifx\XFigheight\undefined\dimen3=0pt\else\dimen3\XFigheight\fi
\divide\dimen3 by 2200
\ifdim\dimen1=0pt\ifdim\dimen3=0pt\dimen1=4143sp\dimen3\dimen1
  \else\dimen1\dimen3\fi\else\ifdim\dimen3=0pt\dimen3\dimen1\fi\fi
\tikzpicture[x=+\dimen1, y=+\dimen3]
{\ifx\XFigu\undefined\catcode`\@11
\def\temp{\alloc@1\dimen\dimendef\insc@unt}\temp\XFigu\catcode`\@12\fi}
\XFigu4143sp
\ifdim\XFigu<0pt\XFigu-\XFigu\fi
\clip(528,-2382) rectangle (2778,-182);
\tikzset{inner sep=+0pt, outer sep=+0pt}
\pgfsetlinewidth{+7.5\XFigu}
\draw (1498,-1156) arc[start angle=+145.0, end angle=+-145.0, radius=+220.4];
\draw (1317,-1156) arc[start angle=+160.7, end angle=+-160.7, radius=+383];
\draw (775,-1156) arc[start angle=+172.03, end angle=+-172.03, radius=+912.3];
\pgfsetdash{}{+0pt}
\draw (595,-1156) arc[start angle=+173.34, end angle=+-173.34, radius=+1090.9];
\pgfsetdash{}{+0pt}
\filldraw  (1678,-1283) circle [radius=+18];
\pgfsetdash{}{+0pt}
\draw (540,-1156) rectangle (1533,-1409);
\pgfsetlinewidth{+30\XFigu}
\pgfsetdash{{+15\XFigu}{+90\XFigu}}{+15\XFigu}
\draw (1227,-1102)--(866,-1102);
\draw (1227,-1463)--(866,-1463);
\draw (1227,-1083)--(866,-1083)--(847,-1083);
\draw (1227,-1482)--(866,-1482);
\draw (1227,-1066)--(866,-1066);
\draw (1227,-1499)--(866,-1499);
\endtikzpicture%

\end{center}
  \caption{The diagram $D/\langle \psi \rangle$ is the quotient of the
diagram $D$ pictured in Figure~\ref{fig:symm} by the action of the
period and is itself the closure of a tangle.}
 \label{fig:tangleclose}
\end{figure}

Clearly, $D/\langle \psi \rangle$ is the closure of a tangle by means of $k$
arcs going around the axis, as shown in Figure~\ref{fig:tangleclose} with the 
axis represented as usual by a dot. We observe that $D/\langle \psi \rangle$ 
must moreover enjoy the following extra properties: it is an alternating 
diagram representing the trivial knot. In particular, it must present a 
Reidemeister I move allowing to reduce the number of crossings according to
\cite{Ba}. The move cannot take place inside the tangle, else $D$ would not be 
minimal. So the loop involved in the move must contain the dot corresponding to 
the axis of the period. The proof will be by induction on $k$, the size of the 
tangle, in a diagram of the form and with the properties just discussed. 

\begin{figure}[ht]
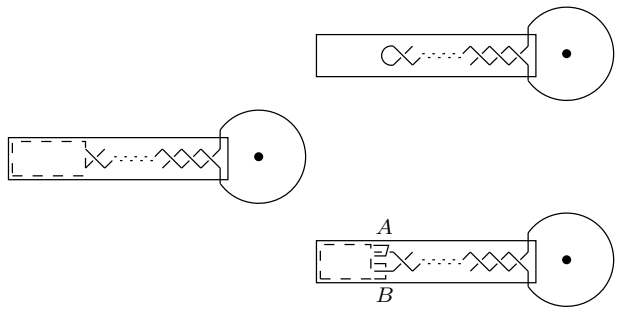

 \begin{center}



\ifx\XFigwidth\undefined\dimen1=0pt\else\dimen1\XFigwidth\fi
\divide\dimen1 by 3603
\ifx\XFigheight\undefined\dimen3=0pt\else\dimen3\XFigheight\fi
\divide\dimen3 by 1803
\ifdim\dimen1=0pt\ifdim\dimen3=0pt\dimen1=4143sp\dimen3\dimen1
  \else\dimen1\dimen3\fi\else\ifdim\dimen3=0pt\dimen3\dimen1\fi\fi
\tikzpicture[x=+\dimen1, y=+\dimen3]
{\ifx\XFigu\undefined\catcode`\@11
\def\temp{\alloc@1\dimen\dimendef\insc@unt}\temp\XFigu\catcode`\@12\fi}
\XFigu4143sp
\ifdim\XFigu<0pt\XFigu-\XFigu\fi
\clip(213,-2037) rectangle (3816,-234);
\tikzset{inner sep=+0pt, outer sep=+0pt}
\pgfsetlinewidth{+7.5\XFigu}
\draw (3301,-362) arc[start angle=+145.0, end angle=+-145.0, radius=+278.3];
\draw (3301,-1591) arc[start angle=+145.0, end angle=+-145.0, radius=+278.3];
\pgfsetdash{}{+0pt}
\draw (2503,-589) arc[start angle=+281.5, end angle=+78.5, radius=+57.7];
\pgfsetdash{}{+0pt}
\draw (1478,-977) arc[start angle=+145.1, end angle=+-145.1, radius=+278];
\filldraw  (3529,-521) circle [radius=+23];
\filldraw  (3529,-1751) circle [radius=+23];
\filldraw  (1706,-1136) circle [radius=+23];
\draw (2503,-476)--(2617,-589);
\draw (2959,-476)--(3073,-589);
\draw (3073,-476)--(3187,-589);
\draw (3187,-476)--(3301,-589);
\draw (2549,-544)--(2503,-589);
\draw (2617,-476)--(2572,-521);
\draw (2663,-544)--(2617,-589);
\draw (3073,-476)--(3028,-521);
\draw (3187,-476)--(3141,-521);
\draw (3301,-476)--(3255,-521);
\draw (3233,-544)--(3187,-589);
\draw (3119,-544)--(3073,-589);
\draw (3005,-544)--(2959,-589);
\draw (2959,-476)--(2914,-521);
\draw (3301,-362)--(3301,-476);
\draw (3301,-589)--(3301,-681);
\pgfsetdash{{+15\XFigu}{+45\XFigu}}{+15\XFigu}
\draw (2914,-521)--(2663,-521);
\draw (2663,-544)--(2914,-544);
\pgfsetdash{}{+0pt}
\draw (2503,-1705)--(2617,-1819);
\draw (2959,-1705)--(3073,-1819);
\draw (3073,-1705)--(3187,-1819);
\draw (3187,-1705)--(3301,-1819);
\draw (2549,-1774)--(2503,-1819);
\draw (2617,-1705)--(2572,-1751);
\draw (2663,-1774)--(2617,-1819);
\draw (3073,-1705)--(3028,-1751);
\draw (3187,-1705)--(3141,-1751);
\draw (3301,-1705)--(3255,-1751);
\draw (3233,-1774)--(3187,-1819);
\draw (3119,-1774)--(3073,-1819);
\draw (3005,-1774)--(2959,-1819);
\draw (2959,-1705)--(2914,-1751);
\draw (3301,-1591)--(3301,-1705);
\draw (3301,-1819)--(3301,-1910);
\pgfsetdash{{+15\XFigu}{+45\XFigu}}{+15\XFigu}
\draw (2914,-1751)--(2663,-1751);
\draw (2663,-1774)--(2914,-1774);
\pgfsetdash{}{+0pt}
\draw (681,-1090)--(795,-1204);
\draw (1136,-1090)--(1250,-1204);
\draw (1250,-1090)--(1364,-1204);
\draw (1364,-1090)--(1478,-1204);
\draw (726,-1159)--(681,-1204);
\draw (795,-1090)--(749,-1136);
\draw (840,-1159)--(795,-1204);
\draw (1250,-1090)--(1205,-1136);
\draw (1364,-1090)--(1319,-1136);
\draw (1478,-1090)--(1433,-1136);
\draw (1410,-1159)--(1364,-1204);
\draw (1296,-1159)--(1250,-1204);
\draw (1182,-1159)--(1136,-1204);
\draw (1136,-1090)--(1091,-1136);
\draw (1478,-977)--(1478,-1090);
\draw (1478,-1204)--(1478,-1295);
\pgfsetdash{{+15\XFigu}{+45\XFigu}}{+15\XFigu}
\draw (1091,-1136)--(840,-1136);
\draw (840,-1159)--(1091,-1159);
\pgfsetdash{}{+0pt}
\draw (225,-1022) rectangle (1524,-1273);
\pgfsetdash{{+60\XFigu}{+60\XFigu}}{++0pt}
\draw (248,-1045) rectangle (681,-1250);
\pgfsetdash{}{+0pt}
\draw (2048,-407) rectangle (3347,-658);
\draw (2048,-1637) rectangle (3347,-1888);
\pgfsetdash{}{+0pt}
\draw (2390,-1819)--(2503,-1819);
\draw (2481,-1705)--(2503,-1705);
\draw (2390,-1705)--(2435,-1705);
\draw (2385,-1665)--(2475,-1665)--(2458,-1728)--(2390,-1728);
\draw (2390,-1774)--(2458,-1774)--(2458,-1796);
\draw (2458,-1842)--(2458,-1865)--(2390,-1865);
\pgfsetdash{{+60\XFigu}{+60\XFigu}}{++0pt}
\draw (2071,-1660) rectangle (2370,-1865);
\pgftext[base,left,at=\pgfqpointxy{2400}{-1600}]
{\fontsize{8}{9.6}\usefont{T1}{ptm}{m}{n}$A$};
\pgftext[base,left,at=\pgfqpointxy{2400}{-2000}]
{\fontsize{8}{9.6}\usefont{T1}{ptm}{m}{n}$B$};
\endtikzpicture%

\end{center}
  \caption{The case $k=1$ where crossings are present. A maximal string of
crossings adjacent to the nugatory one is put in evidence. If there is no other
crossing the situation is as pictured on the top-right, else the two arcs
exiting the maximal string must encounter two distinct crossings labelled $A$
and $B$ as pictured on the bottom-right.}
 \label{fig:induction0}
\end{figure}

If $k=1$ we want to prove that the tangle must consist of a straight 
arc, that is the diagram has no crossings (as in the right-hand side of 
Figure~\ref{fig:quodiag}); in particular, this situation cannot arise under the 
hypotheses of the proposition, for $K$ would be trivial in this case. If that 
is not the case, then the situation is as shown in Figure~\ref{fig:induction0} 
where a maximal string of half-twists adjacent to the loop involved in the
Reidemeister I move is put in evidence. If there are no other 
crossing inside the tangle, then the diagram must be of the form shown on the
top-right of Figure~\ref{fig:induction0}. So assume there are other crossings 
inside the tangle. The situation must be as in the bottom-right of 
Figure~\ref{fig:induction0}, where the crossings $A$ and $B$ must be distinct, 
since the string of half-twists was chosen to be maximal. At this point we see 
that either a nugatory crossing was already present inside the tangle, against 
the hypothesis, or no Reidemeiset I move can be performed, once more 
contradicting the hypothesis, since the knot would not be trivial. 

\begin{figure}[ht]
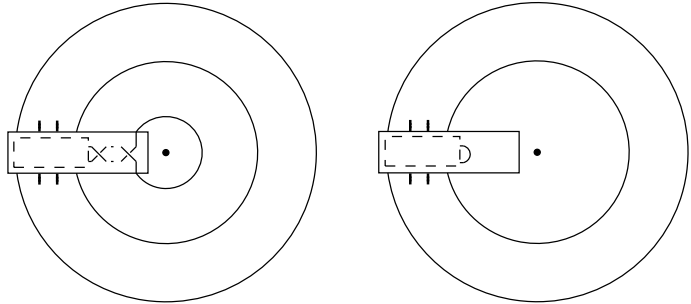

 \begin{center}



\ifx\XFigwidth\undefined\dimen1=0pt\else\dimen1\XFigwidth\fi
\divide\dimen1 by 4046
\ifx\XFigheight\undefined\dimen3=0pt\else\dimen3\XFigheight\fi
\divide\dimen3 by 1804
\ifdim\dimen1=0pt\ifdim\dimen3=0pt\dimen1=4143sp\dimen3\dimen1
  \else\dimen1\dimen3\fi\else\ifdim\dimen3=0pt\dimen3\dimen1\fi\fi
\tikzpicture[x=+\dimen1, y=+\dimen3]
{\ifx\XFigu\undefined\catcode`\@11
\def\temp{\alloc@1\dimen\dimendef\insc@unt}\temp\XFigu\catcode`\@12\fi}
\XFigu4143sp
\ifdim\XFigu<0pt\XFigu-\XFigu\fi
\clip(128,-2024) rectangle (4174,-220);
\tikzset{inner sep=+0pt, outer sep=+0pt}
\pgfsetlinewidth{+7.5\XFigu}
\draw (899,-1002) arc[start angle=+145.1, end angle=+-145.1, radius=+214.7];
\draw (546,-1002) arc[start angle=+166.9, end angle=+-166.9, radius=+543.1];
\draw (193,-1002) arc[start angle=+172.06, end angle=+-172.06, radius=+890.5];
\filldraw  (1075,-1125) circle [radius=+18];
\draw (635,-1090)--(722,-1178);
\draw (810,-1090)--(899,-1178);
\draw (722,-1090)--(687,-1125);
\draw (899,-1090)--(864,-1125);
\draw (845,-1143)--(810,-1178);
\draw (670,-1143)--(635,-1178);
\draw (899,-1002)--(899,-1090);
\draw (899,-1178)--(899,-1248);
\pgfsetdash{{+15\XFigu}{+45\XFigu}}{+15\XFigu}
\draw (810,-1090)--(722,-1090);
\draw (810,-1178)--(722,-1178);
\pgfsetdash{}{+0pt}
\draw (616,-1090)--(635,-1090);
\draw (616,-1178)--(635,-1178);
\pgfsetdash{}{+0pt}
\draw (140,-1002) rectangle (969,-1248);
\pgfsetdash{{+60\XFigu}{+60\XFigu}}{++0pt}
\draw (175,-1037) rectangle (616,-1213);
\pgfsetlinewidth{+30\XFigu}
\pgfsetdash{{+15\XFigu}{+90\XFigu}}{+15\XFigu}
\draw (229,-1266)--(528,-1266);
\draw (229,-1284)--(528,-1284);
\draw (229,-1301)--(528,-1301);
\draw (229,-949)--(528,-949);
\draw (229,-966)--(528,-966);
\draw (229,-984)--(528,-984);
\pgfsetlinewidth{+7.5\XFigu}
\pgfsetdash{}{+0pt}
\draw (2742,-999) arc[start angle=+166.9, end angle=+-166.9, radius=+544.7];
\draw (2389,-999) arc[start angle=+172.05, end angle=+-172.05, radius=+892.6];
\pgfsetdash{}{+0pt}
\draw (2835,-1085) arc[start angle=+79, end angle=+-79, radius=+51];
\pgfsetdash{}{+0pt}
\filldraw  (3273,-1123) circle [radius=+18];
\draw (2335,-999) rectangle (3166,-1246);
\pgfsetlinewidth{+30\XFigu}
\pgfsetdash{{+15\XFigu}{+90\XFigu}}{+15\XFigu}
\draw (2424,-1264)--(2724,-1264);
\draw (2424,-1282)--(2724,-1282);
\draw (2424,-1300)--(2724,-1300);
\draw (2424,-945)--(2724,-945);
\draw (2424,-964)--(2724,-964);
\draw (2424,-981)--(2724,-981);
\pgfsetlinewidth{+7.5\XFigu}
\pgfsetdash{{+60\XFigu}{+60\XFigu}}{++0pt}
\draw (2373,-1029) rectangle (2815,-1205);
\pgfsetdash{}{+0pt}
\draw (2835,-1185)--(2815,-1185);
\draw (2815,-1085)--(2835,-1085);
\endtikzpicture%

\end{center}
  \caption{A new diagram with a tangle of size $k-1$ (shown on the right-hand
side) can be obtained by performing a sequence of Reidemeister I moves that
eliminate a maximal sequence of nugatory crossings (shown on the left-hand
side).}
 \label{fig:induction1}
\end{figure}

We can now assume $k\ge 2$. The situation is thus as shown in 
Figure~\ref{fig:induction1} where, as in the case where $k=1$, a maximal string 
of half-twists adjacent to the loop involved in the Reidemeister I move is put 
in evidence. If we remove the crossings as suggested in the figure by
performing as many Reidemeister I moves as the number of crossings, we get a 
new diagram for the trivial knot, obtained by closing a $(k-1)$-tangle around 
an axis. We claim that this new alternating diagram satisfies the same 
properties as the original one so that we may apply the induction hypothesis to 
finish the proof. The argument is similar to that of the previous case. 

\begin{figure}[ht]
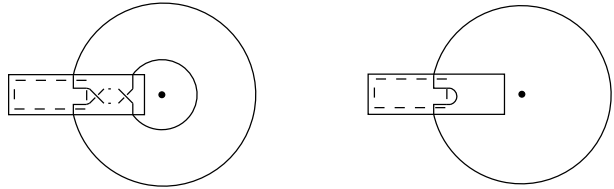

 \begin{center}



\ifx\XFigwidth\undefined\dimen1=0pt\else\dimen1\XFigwidth\fi
\divide\dimen1 by 3587
\ifx\XFigheight\undefined\dimen3=0pt\else\dimen3\XFigheight\fi
\divide\dimen3 by 1110
\ifdim\dimen1=0pt\ifdim\dimen3=0pt\dimen1=4143sp\dimen3\dimen1
  \else\dimen1\dimen3\fi\else\ifdim\dimen3=0pt\dimen3\dimen1\fi\fi
\tikzpicture[x=+\dimen1, y=+\dimen3]
{\ifx\XFigu\undefined\catcode`\@11
\def\temp{\alloc@1\dimen\dimendef\insc@unt}\temp\XFigu\catcode`\@12\fi}
\XFigu4143sp
\ifdim\XFigu<0pt\XFigu-\XFigu\fi
\clip(218,-1342) rectangle (3805,-232);
\tikzset{inner sep=+0pt, outer sep=+0pt}
\pgfsetlinewidth{+7.5\XFigu}
\pgfsetdash{}{+0pt}
\draw (2843,-749) arc[start angle=+79, end angle=+-79, radius=+49.4];
\pgfsetdash{}{+0pt}
\draw (2746,-671) arc[start angle=+167.3, end angle=+-167.3, radius=+531.6];
\draw (965,-668) arc[start angle=+145.1, end angle=+-145.1, radius=+209];
\draw (613,-671) arc[start angle=+167.7, end angle=+-167.7, radius=+546.1];
\filldraw  (3268,-785) circle [radius=+17];
\filldraw  (1137,-788) circle [radius=+17];
\pgfsetdash{{+60\XFigu}{+60\XFigu}}{++0pt}
\draw (2395,-694) rectangle (2824,-865);
\pgfsetdash{}{+0pt}
\draw (2843,-846)--(2824,-846);
\draw (2824,-749)--(2843,-749);
\pgfsetdash{}{+0pt}
\draw (2358,-665) rectangle (3165,-905);
\draw (2746,-671)--(2746,-749)--(2824,-749);
\draw (2824,-846)--(2746,-846)--(2746,-904);
\draw (709,-753)--(794,-839);
\draw (880,-753)--(965,-839);
\draw (794,-753)--(760,-788);
\draw (965,-753)--(931,-788);
\draw (914,-805)--(880,-839);
\draw (743,-805)--(709,-839);
\draw (965,-668)--(965,-753);
\draw (965,-839)--(965,-907);
\pgfsetdash{{+15\XFigu}{+45\XFigu}}{+15\XFigu}
\draw (880,-753)--(794,-753);
\draw (880,-839)--(794,-839);
\pgfsetdash{}{+0pt}
\draw (690,-749)--(709,-753);
\draw (690,-846)--(709,-839);
\pgfsetdash{}{+0pt}
\draw (230,-668) rectangle (1034,-907);
\pgfsetdash{{+60\XFigu}{+60\XFigu}}{++0pt}
\draw (264,-702) rectangle (692,-873);
\pgfsetdash{}{+0pt}
\draw (690,-749)--(613,-749)--(613,-671);
\draw (613,-904)--(613,-846)--(690,-846);
\endtikzpicture%

\end{center}
  \caption{The situation where no other crossing is present. In this case one
must have $k=2$.}
 \label{fig:nocross}
\end{figure}

If there is no other crossing, the only possibility is that $k=2$ and the 
situation is as in Figure~\ref{fig:nocross} which is precisely of the form 
given in Figure~\ref{fig:quodiag}, else we would have a link with more than one 
component. 

We can thus assume that there are other crossings inside the tangle as in 
Figure~\ref{fig:withcrossings}. The main point is that there must be a nugatory 
crossing somewhere in the modified diagram, but such crossing cannot be already 
present in the diagram before modification. Because of that, the crossing must 
be adjacent to both ends of the arc obtained by untwisting the maximal string 
of half-twists. This implies that the crossings $A$ and $B$ adjacent to this
arc (see Figure~\ref{fig:withcrossings} top) must coincide. Note that if $A$ 
and $B$ do not coincide, then a contradiction is reached as in the $k=1$ case, 
albeit now the arc between $A$ and $B$ may go around the dot representing the 
axis (see Figure~\ref{fig:withcrossings} centre-right). So $A$ and $B$ are the 
same crossing and, because the string of half-twists was chosen to be maximal, 
the only possibility is that a strand coming out from the maximal string of 
half-twists goes around the dot representing the axis before coming back to 
cross the second strand, as shown in Figure~\ref{fig:withcrossings} bottom. 
This ensures that the new diagram is of the desired form. 

\begin{figure}[H]
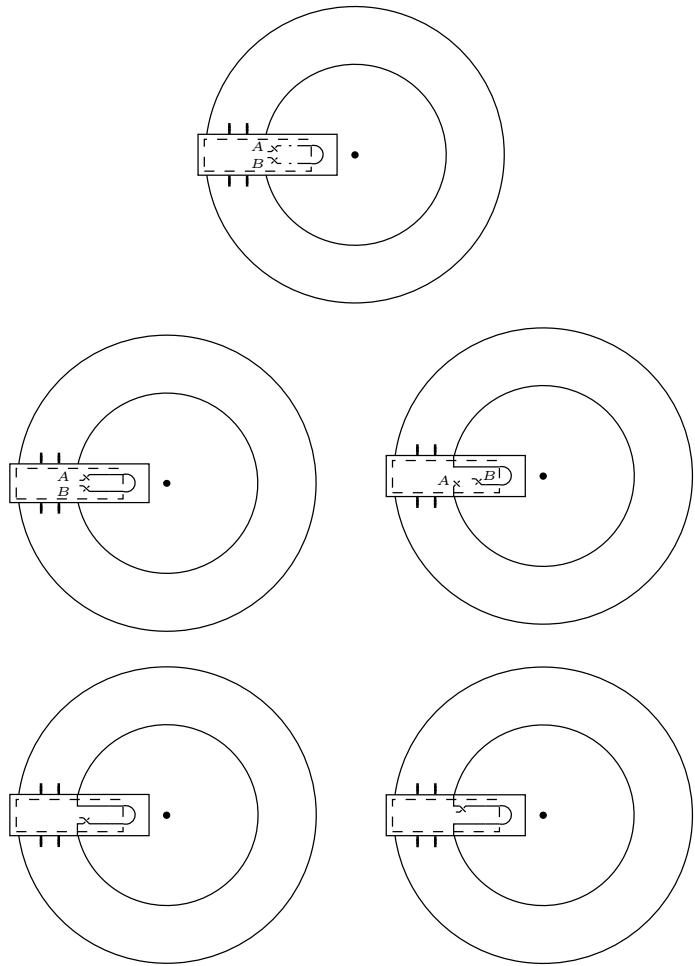

 \begin{center}



\ifx\XFigwidth\undefined\dimen1=0pt\else\dimen1\XFigwidth\fi
\divide\dimen1 by 4061
\ifx\XFigheight\undefined\dimen3=0pt\else\dimen3\XFigheight\fi
\divide\dimen3 by 5729
\ifdim\dimen1=0pt\ifdim\dimen3=0pt\dimen1=4143sp\dimen3\dimen1
  \else\dimen1\dimen3\fi\else\ifdim\dimen3=0pt\dimen3\dimen1\fi\fi
\tikzpicture[x=+\dimen1, y=+\dimen3]
{\ifx\XFigu\undefined\catcode`\@11
\def\temp{\alloc@1\dimen\dimendef\insc@unt}\temp\XFigu\catcode`\@12\fi}
\XFigu4143sp
\ifdim\XFigu<0pt\XFigu-\XFigu\fi
\clip(213,-6077) rectangle (4274,-348);
\tikzset{inner sep=+0pt, outer sep=+0pt}
\pgfsetlinewidth{+7.5\XFigu}
\draw (629,-5060) arc[start angle=+166.8, end angle=+-166.8, radius=+539.7];
\draw (278,-5060) arc[start angle=+172.01, end angle=+-172.01, radius=+884.6];
\draw (895,-5129) arc[start angle=+107, end angle=+-107, radius=+55.4];
\filldraw  (1153,-5184) circle [radius=+18];
\draw (225,-5060) rectangle (1049,-5306);
\pgfsetlinewidth{+30\XFigu}
\pgfsetdash{{+15\XFigu}{+90\XFigu}}{+15\XFigu}
\draw (312,-5323)--(610,-5323);
\draw (312,-5342)--(610,-5342);
\draw (312,-5359)--(610,-5359);
\draw (312,-5008)--(610,-5008);
\draw (312,-5026)--(610,-5026);
\draw (312,-5043)--(610,-5043);
\pgfsetlinewidth{+7.5\XFigu}
\pgfsetdash{{+60\XFigu}{+60\XFigu}}{++0pt}
\draw (263,-5090) rectangle (895,-5282);
\pgfsetdash{}{+0pt}
\draw (895,-5129)--(813,-5129);
\draw (895,-5235)--(813,-5235);
\draw (813,-5235)--(696,-5235)--(660,-5199);
\draw (696,-5199)--(684,-5211);
\draw (636,-5199)--(660,-5199);
\draw (672,-5223)--(660,-5235);
\pgfsetdash{}{+0pt}
\draw (813,-5129)--(625,-5129)--(625,-5058);
\draw (625,-5306)--(625,-5235)--(660,-5235);
\pgfsetdash{}{+0pt}
\draw (2856,-5062) arc[start angle=+166.9, end angle=+-166.9, radius=+539];
\draw (2505,-5062) arc[start angle=+172.07, end angle=+-172.07, radius=+884.5];
\draw (3123,-5129) arc[start angle=+106, end angle=+-106, radius=+55.5];
\filldraw  (3382,-5184) circle [radius=+18];
\draw (2453,-5062) rectangle (3276,-5306);
\pgfsetlinewidth{+30\XFigu}
\pgfsetdash{{+15\XFigu}{+90\XFigu}}{+15\XFigu}
\draw (2541,-5324)--(2838,-5324);
\draw (2541,-5341)--(2838,-5341);
\draw (2541,-5360)--(2838,-5360);
\draw (2541,-5008)--(2838,-5008);
\draw (2541,-5026)--(2838,-5026);
\draw (2541,-5044)--(2838,-5044);
\pgfsetlinewidth{+7.5\XFigu}
\pgfsetdash{{+60\XFigu}{+60\XFigu}}{++0pt}
\draw (2490,-5090) rectangle (3123,-5282);
\pgfsetdash{}{+0pt}
\draw (3123,-5129)--(3041,-5129);
\draw (3123,-5236)--(3041,-5236);
\draw (3041,-5129)--(2923,-5129)--(2912,-5141);
\draw (2923,-5165)--(2888,-5129);
\draw (2865,-5165)--(2888,-5165)--(2900,-5153);
\pgfsetdash{}{+0pt}
\draw (2852,-5306)--(2852,-5236)--(3041,-5236);
\draw (2852,-5059)--(2852,-5129)--(2888,-5129);
\pgfsetdash{}{+0pt}
\draw (2856,-3037) arc[start angle=+166.9, end angle=+-166.9, radius=+539.6];
\draw (2506,-3037) arc[start angle=+172.03, end angle=+-172.03, radius=+884];
\draw (3123,-3105) arc[start angle=+106, end angle=+-106, radius=+55.1];
\filldraw  (3382,-3160) circle [radius=+18];
\draw (2453,-3037) rectangle (3276,-3282);
\pgfsetlinewidth{+30\XFigu}
\pgfsetdash{{+15\XFigu}{+90\XFigu}}{+15\XFigu}
\draw (2540,-3299)--(2838,-3299);
\draw (2540,-3317)--(2838,-3317);
\draw (2540,-3335)--(2838,-3335);
\draw (2540,-2985)--(2838,-2985);
\draw (2540,-3001)--(2838,-3001);
\draw (2540,-3019)--(2838,-3019);
\pgfsetlinewidth{+7.5\XFigu}
\pgfsetdash{}{+0pt}
\draw (3123,-3105)--(2852,-3105);
\draw (3123,-3211)--(3041,-3211);
\draw (2848,-3105)--(2856,-3105);
\pgfsetdash{{+60\XFigu}{+60\XFigu}}{++0pt}
\draw (2490,-3067) rectangle (3123,-3259);
\pgfsetdash{}{+0pt}
\draw (2852,-3105)--(2852,-3034);
\draw (3041,-3211)--(3017,-3211)--(2982,-3176);
\pgfsetdash{}{+0pt}
\draw (2959,-3176)--(2982,-3176);
\draw (2993,-3199)--(2982,-3211);
\draw (3017,-3176)--(3005,-3188);
\pgfsetdash{}{+0pt}
\draw (2852,-3282)--(2852,-3223)--(2863,-3211);
\pgfsetdash{}{+0pt}
\draw (2888,-3223)--(2852,-3188);
\pgfsetdash{}{+0pt}
\pgfsetdash{}{+0pt}
\draw (2888,-3188)--(2876,-3199);
\pgfsetdash{}{+0pt}
\draw (629,-3087) arc[start angle=+167.6, end angle=+-167.6, radius=+538];
\draw (278,-3087) arc[start angle=+172.49, end angle=+-172.49, radius=+883.6];
\draw (895,-3151) arc[start angle=+110, end angle=+-110, radius=+53.2];
\filldraw  (1154,-3203) ellipse [x radius=+18,y radius=+17];
\draw (225,-3087) rectangle (1049,-3318);
\pgfsetlinewidth{+30\XFigu}
\pgfsetdash{{+15\XFigu}{+90\XFigu}}{+15\XFigu}
\draw (313,-3336)--(610,-3336);
\draw (313,-3352)--(610,-3352);
\draw (313,-3369)--(610,-3369);
\draw (313,-3035)--(610,-3035);
\draw (313,-3053)--(610,-3053);
\draw (313,-3069)--(610,-3069);
\pgfsetlinewidth{+7.5\XFigu}
\pgfsetdash{{+60\XFigu}{+60\XFigu}}{++0pt}
\draw (263,-3114) rectangle (895,-3296);
\pgfsetdash{}{+0pt}
\draw (895,-3151)--(813,-3151);
\draw (895,-3251)--(813,-3251);
\draw (813,-3251)--(696,-3251)--(660,-3218);
\draw (813,-3151)--(696,-3151)--(684,-3163);
\draw (696,-3185)--(660,-3151);
\draw (637,-3185)--(660,-3185)--(672,-3173);
\draw (696,-3218)--(684,-3229);
\draw (637,-3218)--(660,-3218);
\draw (672,-3240)--(660,-3251);
\draw (1742,-1120) arc[start angle=+166.9, end angle=+-166.9, radius=+540];
\draw (1391,-1120) arc[start angle=+172.07, end angle=+-172.07, radius=+884.5];
\draw (2009,-1188) arc[start angle=+106, end angle=+-106, radius=+55.1];
\filldraw  (2268,-1243) circle [radius=+18];
\draw (1339,-1120) rectangle (2162,-1364);
\pgfsetlinewidth{+30\XFigu}
\pgfsetdash{{+15\XFigu}{+90\XFigu}}{+15\XFigu}
\draw (1427,-1381)--(1724,-1381);
\draw (1427,-1399)--(1724,-1399);
\draw (1427,-1417)--(1724,-1417);
\draw (1427,-1068)--(1724,-1068);
\draw (1427,-1085)--(1724,-1085);
\draw (1427,-1103)--(1724,-1103);
\pgfsetlinewidth{+7.5\XFigu}
\pgfsetdash{{+60\XFigu}{+60\XFigu}}{++0pt}
\draw (1376,-1150) rectangle (2009,-1340);
\pgfsetdash{}{+0pt}
\draw (2009,-1188)--(1927,-1188);
\draw (2009,-1294)--(1927,-1294);
\draw (1832,-1294)--(1809,-1294)--(1773,-1259);
\draw (1832,-1188)--(1809,-1188)--(1797,-1200);
\draw (1809,-1223)--(1773,-1188);
\draw (1750,-1223)--(1773,-1223)--(1786,-1212);
\draw (1809,-1259)--(1797,-1270);
\draw (1750,-1259)--(1773,-1259);
\draw (1786,-1282)--(1773,-1294);
\pgfsetdash{{+15\XFigu}{+45\XFigu}}{+15\XFigu}
\draw (1832,-1294)--(1939,-1294);
\draw (1832,-1188)--(1927,-1188);
\pgftext[base,left,at=\pgfqpointxy{3020}{-3185}]
{\fontsize{5}{6}\usefont{T1}{ptm}{m}{n}$B$};
\pgftext[base,left,at=\pgfqpointxy{2750}{-3211}]
{\fontsize{5}{6}\usefont{T1}{ptm}{m}{n}$A$};
\pgftext[base,left,at=\pgfqpointxy{500}{-3280}]
{\fontsize{5}{6}\usefont{T1}{ptm}{m}{n}$B$};
\pgftext[base,left,at=\pgfqpointxy{500}{-3190}]
{\fontsize{5}{6}\usefont{T1}{ptm}{m}{n}$A$};
\pgftext[base,left,at=\pgfqpointxy{1650}{-1320}]
{\fontsize{5}{4.8}\usefont{T1}{ptm}{m}{n}$B$};
\pgftext[base,left,at=\pgfqpointxy{1650}{-1220}]
{\fontsize{5}{4.8}\usefont{T1}{ptm}{m}{n}$A$};
\endtikzpicture%

\end{center}
  \caption{The diagram when other crossings are present. The situation in which
the crossings adjacent to the arc obtained after untwisting the maximal string
of half-twists are not the same is pictured in the middle, while the case where
they coincide is shown on the bottom.}
 \label{fig:withcrossings}
\end{figure}

We are now left to understand why the original diagram is also of the desired
form. In principle, the sequence of crossing removed in the process could be
inserted in two distinct ways (see Figure~\ref{fig:withcrossings} bottom). It 
is thus enough to show that the two situations are isotopic via an isotopy 
that preserves the dot, i.e. the projection of the axis of rotation. This is 
clear from the figure and the discussion preceeding the proof.
\end{proof}

We remark that in the proof of the proposition, one might want to change the
diagram by an isotopy so that another nugatory crossing adjacent to the loop
encircling the fixed-point of $Fix(\psi)/\langle \psi \rangle$ at infinity 
appears. However, for the knot to be trivial both the innermost and outermost 
loops must be adjacent to a nugatory crossing.


\section{Symmetry of the quotients}\label{s:symmetry}

In this section we study the link having a diagram of the form determined in
the previous section. 

\begin{Proposition}\label{p:symdiag}
Let $K$ be a non trivial prime alternating knot with period $\psi$ of oder $n$.
Assume that $\psi$ is visible on a minimal diagram $D$ for $K$ and that
$K/\langle \psi \rangle$ is the trivial knot. Then there is homeomorphism of 
the $3$-sphere that exchanges the two components of the link 
$(K\cup Fix(\psi))/\langle \psi \rangle$.
\end{Proposition}

\begin{proof}
By hypothesis, Proposition~\ref{p:quodiag} applies and we know that 
$(K\cup Fix(\psi))/\langle \psi \rangle$ admits a diagram as in
Figure~\ref{fig:quodiag} where the second component is represented by the
central dot. The second component of the link is a trivial knot that encircles
the $k$ arcs that close up the tangle. This is pictured in the left-hand side
of Figure~\ref{fig:linksym}, where an isotopy was performed so that the
sequences of half-twists appear alternately on the right and left.

\begin{figure}[ht]
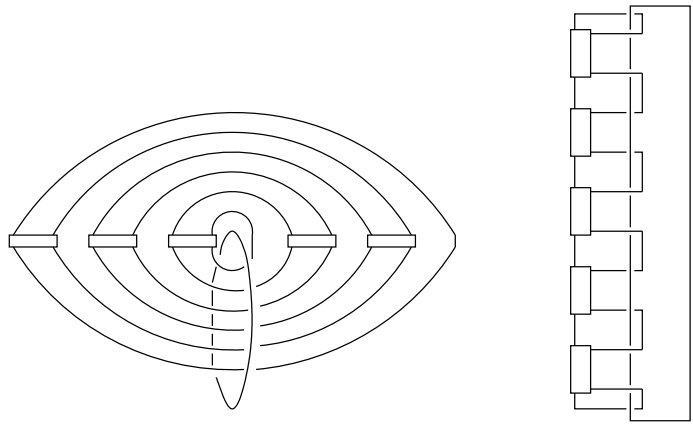

 \begin{center}



\ifx\XFigwidth\undefined\dimen1=0pt\else\dimen1\XFigwidth\fi
\divide\dimen1 by 4054
\ifx\XFigheight\undefined\dimen3=0pt\else\dimen3\XFigheight\fi
\divide\dimen3 by 2499
\ifdim\dimen1=0pt\ifdim\dimen3=0pt\dimen1=4143sp\dimen3\dimen1
  \else\dimen1\dimen3\fi\else\ifdim\dimen3=0pt\dimen3\dimen1\fi\fi
\tikzpicture[x=+\dimen1, y=+\dimen3]
{\ifx\XFigu\undefined\catcode`\@11
\def\temp{\alloc@1\dimen\dimendef\insc@unt}\temp\XFigu\catcode`\@12\fi}
\XFigu4143sp
\ifdim\XFigu<0pt\XFigu-\XFigu\fi
\clip(168,-2622) rectangle (4222,-123);
\tikzset{inner sep=+0pt, outer sep=+0pt}
\pgfsetlinewidth{+7.5\XFigu}
\draw (439,-1502) arc[start angle=+150.06, end angle=+29.94, radius=+1223.8];
\draw (675,-1502) arc[start angle=+151.93, end angle=+28.07, radius=+935];
\draw (911,-1502) arc[start angle=+155.24, end angle=+24.76, radius=+648.6];
\draw (1146,-1502) arc[start angle=+162.5, end angle=+17.5, radius=+370.7];
\draw (1382,-1502) arc[start angle=+190.1, end angle=+-10.1, radius=+119.9];
\draw (204,-1502) arc[start angle=+148.40, end angle=+31.60, radius=+1535.7];
\draw (911,-1573) arc[start angle=+-155.53, end angle=+-82.26, radius=+653.7];
\draw (1665,-1926) arc[start angle=+-77.01, end angle=+-23.43, radius=+612.1];
\draw (675,-1573) arc[start angle=+-151.43, end angle=+-86.46, radius=+952.1];
\draw (1665,-2044) arc[start angle=+-81.74, end angle=+-27.23, radius=+885.3];
\draw (439,-1573) arc[start angle=+-149.85, end angle=+-87.07, radius=+1235];
\draw (1665,-2162) arc[start angle=+-84.24, end angle=+-29.06, radius=+1156.6];
\draw (204,-1573) arc[start angle=+-148.93, end angle=+-87.38, radius=+1514];
\draw (1641,-2304) arc[start angle=+-84.67, end angle=+-31.74, radius=+1556.4];
\draw (1382,-1573) arc[start angle=+169.9, end angle=+305.9, radius=+119.7];
\draw (1641,-1809) arc[start angle=+-71.7, end angle=+-12.1, radius=+319];
\draw (1146,-1573) arc[start angle=+-159.4, end angle=+-83.4, radius=+403.3];
\draw (180,-1502) rectangle (463,-1573);
\draw (651,-1502) rectangle (934,-1573);
\draw (1123,-1502) rectangle (1405,-1573);
\draw (1830,-1502) rectangle (2113,-1573);
\draw (2301,-1502) rectangle (2584,-1573);
\draw (1618,-1502)--(1618,-1644);
\draw (3503,-276) rectangle (3621,-559);
\draw (3503,-748) rectangle (3621,-1031);
\draw (3503,-1691) rectangle (3621,-1974);
\draw (3503,-1219) rectangle (3621,-1502);
\draw (3503,-2162) rectangle (3621,-2445);
\draw (3527,-559)--(3527,-748);
\draw (3527,-1031)--(3527,-1219);
\draw (3527,-1502)--(3527,-1691);
\draw (3527,-1974)--(3527,-2162);
\draw (3621,-300)--(3927,-300);
\draw (3621,-536)--(3927,-536);
\draw (3621,-1243)--(3927,-1243);
\draw (3621,-1479)--(3927,-1479);
\draw (3621,-2186)--(3927,-2186);
\draw (3621,-2421)--(3927,-2421);
\draw (3621,-771)--(3833,-771);
\draw (3621,-1007)--(3833,-1007);
\draw (3621,-1714)--(3833,-1714);
\draw (3621,-1950)--(3833,-1950);
\draw (3527,-276)--(3527,-182)--(3833,-182);
\draw (3527,-2445)--(3527,-2539)--(3833,-2539);
\draw (3856,-324)--(3856,-512);
\draw (3856,-1266)--(3856,-1455);
\draw (3856,-2209)--(3856,-2398);
\draw (3856,-559)--(3856,-1219);
\draw (3856,-1502)--(3856,-2162);
\draw (3880,-182)--(3927,-182)--(3927,-300);
\draw (3927,-2421)--(3927,-2539)--(3880,-2539);
\draw (3880,-1950)--(3927,-1950)--(3927,-2186);
\draw (3880,-1007)--(3927,-1007)--(3927,-1243);
\draw (3927,-536)--(3927,-771)--(3880,-771);
\draw (3927,-1479)--(3927,-1714)--(3880,-1714);
\draw (3856,-276)--(3856,-135)--(4210,-135)--(4210,-2610)--(3856,-2610)--(3856,-2445);
\draw (2820,-1502)--(2820,-1573);
\draw (1429,-1620)--(1429,-1619)--(1429,-1615)--(1430,-1607)--(1431,-1597)--(1433,-1587)
  --(1435,-1577)--(1437,-1567)--(1441,-1557)--(1445,-1546)--(1451,-1533)--(1458,-1519)
  --(1466,-1506)--(1475,-1495)--(1483,-1486)--(1492,-1480)--(1500,-1478)--(1506,-1480)
  --(1513,-1483)--(1520,-1489)--(1527,-1497)--(1534,-1507)--(1541,-1519)--(1547,-1530)
  --(1553,-1542)--(1558,-1554)--(1562,-1565)--(1566,-1576)--(1570,-1587)--(1574,-1598)
  --(1578,-1610)--(1581,-1622)--(1584,-1635)--(1587,-1648)--(1589,-1662)--(1592,-1676)
  --(1594,-1691)--(1596,-1702)--(1597,-1714)--(1599,-1727)--(1600,-1741)--(1602,-1756)
  --(1604,-1771)--(1605,-1788)--(1607,-1804)--(1608,-1821)--(1610,-1837)--(1611,-1854)
  --(1612,-1869)--(1613,-1884)--(1614,-1899)--(1615,-1914)--(1615,-1929)--(1616,-1944)
  --(1616,-1959)--(1617,-1975)--(1617,-1992)--(1617,-2008)--(1617,-2024)--(1617,-2040)
  --(1616,-2055)--(1616,-2070)--(1615,-2084)--(1615,-2098)--(1614,-2111)--(1613,-2126)
  --(1612,-2142)--(1610,-2158)--(1609,-2175)--(1607,-2192)--(1605,-2209)--(1603,-2226)
  --(1600,-2243)--(1598,-2259)--(1595,-2274)--(1593,-2289)--(1590,-2304)--(1587,-2318)
  --(1584,-2333)--(1581,-2348)--(1578,-2364)--(1574,-2380)--(1570,-2395)--(1566,-2411)
  --(1562,-2425)--(1558,-2439)--(1554,-2451)--(1551,-2462)--(1547,-2473)--(1543,-2484)
  --(1538,-2495)--(1533,-2505)--(1529,-2514)--(1523,-2522)--(1518,-2528)--(1514,-2533)
  --(1509,-2537)--(1504,-2539)--(1500,-2539)--(1496,-2539)--(1491,-2537)--(1486,-2533)
  --(1482,-2529)--(1476,-2523)--(1471,-2515)--(1466,-2507)--(1462,-2499)--(1457,-2490)
  --(1453,-2480)--(1448,-2470)--(1444,-2459)--(1439,-2445)--(1433,-2429)--(1426,-2410)
  --(1419,-2390)--(1412,-2372)--(1408,-2358)--(1405,-2352)--(1405,-2351);
\draw (1405,-1691)--(1382,-1785);
\draw (1382,-1832)--(1382,-1926);
\draw (1382,-1974)--(1382,-2044);
\draw (1382,-2091)--(1382,-2162);
\draw (1382,-2209)--(1382,-2280);
\endtikzpicture%

\end{center}
  \caption{Two diagrams of the link $(K\cup Fix(\psi))/\langle \psi \rangle$ as
seen from above and from the side, respectively.}
 \label{fig:linksym}
\end{figure}

To prove the assertion it is suitable to modify the given diagram as in
Figure~\ref{fig:linksym} right. One can think of the two diagrams as
projections of the link from above and from the side respectively. To visualise 
how to pass from the diagram pictured on the left-hand side to the one on the 
right-hand side it is convenient to imagine that the central part of the 
diagram on the left is at the top and the link is located lower and lower down 
the further we move away from the centre. The sequences of half-twists (i.e. 
the boxes) on the right are moved to the left by making them pass behind the 
second component of the link. All sequences of half-twists are now arranged 
vertically, rather than horizontally. 

To prove the proposition it is then enough to show that one can
transfer the crossings of the $K/\langle \psi \rangle$ component onto the 
$Fix(\psi)/\langle \psi \rangle$ component so that afterwards the second
component looks like the first one used to. This can be done as explained in
Figure~\ref{fig:pushing}. 

\begin{figure}[ht]
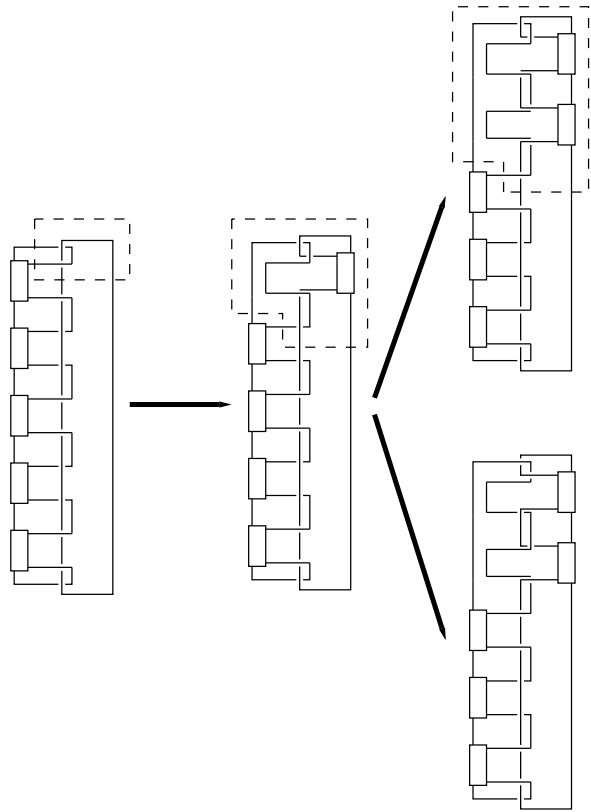

 \begin{center}



\ifx\XFigwidth\undefined\dimen1=0pt\else\dimen1\XFigwidth\fi
\divide\dimen1 by 3444
\ifx\XFigheight\undefined\dimen3=0pt\else\dimen3\XFigheight\fi
\divide\dimen3 by 4812
\ifdim\dimen1=0pt\ifdim\dimen3=0pt\dimen1=4143sp\dimen3\dimen1
  \else\dimen1\dimen3\fi\else\ifdim\dimen3=0pt\dimen3\dimen1\fi\fi
\tikzpicture[x=+\dimen1, y=+\dimen3]
{\ifx\XFigu\undefined\catcode`\@11
\def\temp{\alloc@1\dimen\dimendef\insc@unt}\temp\XFigu\catcode`\@12\fi}
\XFigu4143sp
\ifdim\XFigu<0pt\XFigu-\XFigu\fi
\pgfdeclarearrow{
  name = xfiga1,
  parameters = {
    \the\pgfarrowlinewidth \the\pgfarrowlength \the\pgfarrowwidth\ifpgfarrowopen o\fi},
  defaults = {
	  line width=+7.5\XFigu, length=+120\XFigu, width=+60\XFigu},
  setup code = {
    \dimen7 2.1\pgfarrowlength\pgfmathveclen{\the\dimen7}{\the\pgfarrowwidth}
    \dimen7 2\pgfarrowwidth\pgfmathdivide{\pgfmathresult}{\the\dimen7}
    \dimen7 \pgfmathresult\pgfarrowlinewidth
    \pgfarrowssettipend{+\dimen7}
    \pgfarrowssetbackend{+-\pgfarrowlength}
    \dimen9 -\pgfarrowlength\advance\dimen9 by-0.45\pgfarrowlinewidth
    \pgfarrowssetlineend{+\dimen9}
    \dimen9 -\pgfarrowlength\advance\dimen9 by-0.5\pgfarrowlinewidth
    \pgfarrowssetvisualbackend{+\dimen9}
    \pgfarrowshullpoint{+\dimen7}{+0pt}
    \pgfarrowsupperhullpoint{+-\pgfarrowlength}{+0.5\pgfarrowwidth}
    \pgfarrowssavethe\pgfarrowlinewidth
    \pgfarrowssavethe\pgfarrowlength
    \pgfarrowssavethe\pgfarrowwidth
  },
  drawing code = {\pgfsetdash{}{+0pt}
    \ifdim\pgfarrowlinewidth=\pgflinewidth\else\pgfsetlinewidth{+\pgfarrowlinewidth}\fi
    \pgfpathmoveto{\pgfqpoint{-\pgfarrowlength}{-0.5\pgfarrowwidth}}
    \pgfpathlineto{\pgfqpoint{0pt}{0pt}}
    \pgfpathlineto{\pgfqpoint{-\pgfarrowlength}{0.5\pgfarrowwidth}}
    \pgfpathclose
    \ifpgfarrowopen\pgfusepathqstroke\else\pgfsetfillcolor{.}
	\ifdim\pgfarrowlinewidth>0pt\pgfusepathqfillstroke\else\pgfusepathqfill\fi\fi
  }
}
\clip(168,-4935) rectangle (3612,-123);
\tikzset{inner sep=+0pt, outer sep=+0pt}
\pgfsetlinewidth{+7.5\XFigu}
\draw (1588,-2026) rectangle (1689,-2267);
\draw (1588,-2831) rectangle (1689,-3072);
\draw (1588,-2428) rectangle (1689,-2670);
\draw (1588,-3233) rectangle (1689,-3475);
\draw (1608,-1865)--(1608,-2026);
\draw (1608,-2267)--(1608,-2428);
\draw (1608,-2670)--(1608,-2831);
\draw (1608,-3072)--(1608,-3233);
\draw (1689,-1664)--(1950,-1664);
\draw (1689,-1845)--(1950,-1845);
\draw (1689,-2449)--(1950,-2449);
\draw (1689,-2650)--(1950,-2650);
\draw (1689,-3253)--(1950,-3253);
\draw (1689,-3454)--(1950,-3454);
\draw (1689,-2046)--(1870,-2046);
\draw (1689,-2247)--(1870,-2247);
\draw (1689,-2851)--(1870,-2851);
\draw (1689,-3052)--(1870,-3052);
\draw (1608,-3475)--(1608,-3555)--(1870,-3555);
\draw (1890,-2469)--(1890,-2630);
\draw (1890,-3273)--(1890,-3434);
\draw (1890,-1865)--(1890,-2428);
\draw (1890,-2670)--(1890,-3233);
\draw (1910,-1543)--(1950,-1543)--(1950,-1664);
\draw (1950,-3454)--(1950,-3555)--(1910,-3555);
\draw (1910,-3052)--(1950,-3052)--(1950,-3253);
\draw (1910,-2247)--(1950,-2247)--(1950,-2449);
\draw (1950,-1845)--(1950,-2046)--(1910,-2046);
\draw (1950,-2650)--(1950,-2851)--(1910,-2851);
\draw (1608,-1865)--(1608,-1543)--(1870,-1543);
\pgfsetdash{}{+0pt}
\draw (1890,-1624)--(1890,-1503)--(2192,-1503)--(2192,-1604);
\pgfsetdash{}{+0pt}
\draw (2111,-1604) rectangle (2212,-1845);
\draw (2192,-1845)--(2192,-3615)--(1890,-3615)--(1890,-3475);
\pgfsetdash{}{+0pt}
\draw (1689,-1664)--(1689,-1845);
\draw (2111,-1624)--(1970,-1624);
\draw (1890,-1624)--(1930,-1624);
\draw (2111,-1825)--(1890,-1825);
\pgfsetdash{}{+0pt}
\draw (2896,-1523) rectangle (2996,-1765);
\draw (2896,-1121) rectangle (2996,-1362);
\draw (2896,-1925) rectangle (2996,-2167);
\draw (2916,-557)--(2916,-940);
\draw (2916,-940)--(2916,-1121);
\draw (2916,-1362)--(2916,-1523);
\draw (2916,-1765)--(2916,-1925);
\draw (2996,-356)--(3258,-356);
\draw (2996,-537)--(3258,-537);
\draw (2996,-1141)--(3258,-1141);
\draw (2996,-1342)--(3258,-1342);
\draw (2996,-1946)--(3258,-1946);
\draw (2996,-2147)--(3258,-2147);
\draw (2996,-759)--(3258,-759);
\draw (2996,-920)--(3258,-920);
\draw (2996,-1543)--(3178,-1543);
\draw (2996,-1744)--(3178,-1744);
\draw (2916,-2167)--(2916,-2247)--(3178,-2247);
\draw (3198,-1161)--(3198,-1322);
\draw (3198,-1966)--(3198,-2127);
\draw (3198,-1362)--(3198,-1925);
\draw (3218,-236)--(3258,-236)--(3258,-356);
\draw (3258,-2147)--(3258,-2247)--(3218,-2247);
\draw (3218,-1744)--(3258,-1744)--(3258,-1946);
\draw (3258,-1342)--(3258,-1543)--(3218,-1543);
\draw (2916,-557)--(2916,-236)--(3178,-236);
\pgfsetdash{}{+0pt}
\draw (3198,-316)--(3198,-195)--(3499,-195)--(3499,-296);
\pgfsetdash{}{+0pt}
\draw (3419,-296) rectangle (3520,-537);
\draw (3499,-960)--(3499,-2308)--(3198,-2308)--(3198,-2167);
\pgfsetdash{}{+0pt}
\draw (2996,-356)--(2996,-537);
\draw (3419,-316)--(3278,-316);
\draw (3198,-316)--(3238,-316);
\draw (3419,-517)--(3198,-517);
\pgfsetdash{}{+0pt}
\draw (2996,-759)--(2996,-920);
\draw (3419,-718) rectangle (3520,-960);
\pgfsetdash{}{+0pt}
\draw (3499,-537)--(3499,-718);
\draw (3419,-739)--(3198,-739);
\draw (3419,-940)--(3198,-940);
\pgfsetdash{}{+0pt}
\draw (3258,-960)--(3258,-1141);
\pgfsetdash{}{+0pt}
\draw (3198,-940)--(3198,-1121);
\pgfsetdash{}{+0pt}
\draw (3258,-537)--(3258,-718);
\pgfsetdash{}{+0pt}
\draw (3198,-557)--(3198,-739);
\pgfsetdash{}{+0pt}
\draw (180,-1651) rectangle (281,-1892);
\draw (180,-2053) rectangle (281,-2294);
\draw (180,-2858) rectangle (281,-3099);
\draw (180,-2455) rectangle (281,-2697);
\draw (180,-3260) rectangle (281,-3502);
\draw (200,-1892)--(200,-2053);
\draw (200,-2294)--(200,-2455);
\draw (200,-2697)--(200,-2858);
\draw (200,-3099)--(200,-3260);
\draw (281,-1671)--(542,-1671);
\draw (281,-1872)--(542,-1872);
\draw (281,-2476)--(542,-2476);
\draw (281,-2677)--(542,-2677);
\draw (281,-3280)--(542,-3280);
\draw (281,-3481)--(542,-3481);
\draw (281,-2073)--(462,-2073);
\draw (281,-2274)--(462,-2274);
\draw (281,-2878)--(462,-2878);
\draw (281,-3079)--(462,-3079);
\draw (200,-1651)--(200,-1570)--(462,-1570);
\draw (200,-3502)--(200,-3582)--(462,-3582);
\draw (482,-1691)--(482,-1852);
\draw (482,-2496)--(482,-2657);
\draw (482,-3300)--(482,-3461);
\draw (482,-1892)--(482,-2455);
\draw (482,-2697)--(482,-3260);
\draw (502,-1570)--(542,-1570)--(542,-1671);
\draw (542,-3481)--(542,-3582)--(502,-3582);
\draw (502,-3079)--(542,-3079)--(542,-3280);
\draw (502,-2274)--(542,-2274)--(542,-2476);
\draw (542,-1872)--(542,-2073)--(502,-2073);
\draw (542,-2677)--(542,-2878)--(502,-2878);
\draw (482,-1651)--(482,-1530)--(784,-1530)--(784,-3642)--(482,-3642)--(482,-3502);
\pgfsetdash{{+60\XFigu}{+60\XFigu}}{++0pt}
\draw (321,-1402)--(321,-1704)--(321,-1765)--(381,-1765)--(884,-1765)--(884,-1402)--cycle;
\draw (1488,-1402)--(1488,-1966)--(1789,-1966)--(1789,-2167)--(2292,-2167)--(2292,-1402)--cycle;
\draw (3097,-1241)--(3097,-1060)--(2795,-1060)--(2795,-135)--(3600,-135)--(3600,-1241)--cycle;
\pgfsetdash{}{+0pt}
\draw (2896,-4138) rectangle (2996,-4380);
\draw (2896,-3736) rectangle (2996,-3977);
\draw (2896,-4541) rectangle (2996,-4782);
\draw (2916,-3173)--(2916,-3555);
\draw (2916,-3555)--(2916,-3736);
\draw (2916,-3977)--(2916,-4138);
\draw (2916,-4380)--(2916,-4541);
\draw (2996,-3756)--(3258,-3756);
\draw (2996,-3957)--(3258,-3957);
\draw (2996,-4561)--(3258,-4561);
\draw (2996,-4762)--(3258,-4762);
\draw (2996,-3535)--(3258,-3535);
\draw (2996,-4159)--(3178,-4159);
\draw (2996,-4360)--(3178,-4360);
\draw (2916,-4782)--(2916,-4863)--(3178,-4863);
\draw (3198,-3776)--(3198,-3937);
\draw (3198,-4581)--(3198,-4742);
\draw (3198,-3977)--(3198,-4541);
\draw (3258,-4762)--(3258,-4863)--(3218,-4863);
\draw (3218,-4360)--(3258,-4360)--(3258,-4561);
\draw (3258,-3957)--(3258,-4159)--(3218,-4159);
\draw (3419,-2911) rectangle (3520,-3153);
\draw (3499,-3575)--(3499,-4923)--(3198,-4923)--(3198,-4782);
\pgfsetdash{}{+0pt}
\draw (2996,-2972)--(2996,-3153);
\pgfsetdash{}{+0pt}
\draw (2996,-3374)--(2996,-3535);
\draw (3419,-3334) rectangle (3520,-3575);
\pgfsetdash{}{+0pt}
\draw (3499,-3153)--(3499,-3334);
\draw (3419,-3354)--(3278,-3354);
\draw (3419,-3555)--(3198,-3555);
\pgfsetdash{}{+0pt}
\draw (3258,-3575)--(3258,-3756);
\pgfsetdash{}{+0pt}
\draw (3198,-3555)--(3198,-3736);
\pgfsetdash{}{+0pt}
\draw (2996,-2972)--(3258,-2972);
\draw (2916,-3173)--(2916,-2851)--(3258,-2851)--(3258,-2911);
\pgfsetdash{}{+0pt}
\draw (3198,-2831)--(3198,-2811)--(3499,-2811)--(3499,-2911);
\draw (3419,-2931)--(3198,-2931)--(3198,-2871);
\pgfsetdash{}{+0pt}
\draw (3258,-2951)--(3258,-2972);
\draw (2996,-3153)--(3178,-3153);
\draw (2996,-3374)--(3258,-3374)--(3258,-3153)--(3218,-3153);
\pgfsetdash{}{+0pt}
\draw (3419,-3133)--(3198,-3133)--(3198,-3354)--(3238,-3354);
\pgfsetlinewidth{+30\XFigu}
\pgfsetdash{}{+0pt}
\pgfsetarrows{[line width=7.5\XFigu, width=27\XFigu, length=54\XFigu]}
\pgfsetarrowsend{xfiga1}
\draw (884,-2509)--(1488,-2509);
\draw (2333,-2469)--(2755,-1262);
\draw (2333,-2569)--(2755,-3917);
\endtikzpicture%

\end{center}
  \caption{Moving crossings from the first component to the second one and the
result according to the parity of the number of crossings.}
 \label{fig:pushing}
\end{figure}

Starting from the top, we consider the first sequence of crossings that appears
and transfer the crossings on the other side by performing flypes on the tangle
in the dashed box drawn on the left-hand side of Figure~\ref{fig:pushing}. The 
result is the central diagram in Figure~\ref{fig:pushing}. We can now repeat 
the same process with the next sequence of half-twists by performing flypes 
with respect to the tangle delimited by the dashed curve in the central 
diagram. It turns out that the result depends on the parity of the number of 
flypes we perform, i.e. crossings in the box. The result in the case of an even 
number of crossings is shown on the top right-hand side of 
Figure~\ref{fig:pushing}, while the case of an odd number is given on the 
bottom. One can see that the structure is exactly the same in the case of an 
even number of crossings, but that is not the case if the number is odd. This
phenomenon occurs only from the second sequence of half-twists onward, though,
as the parity of the number of crossings in the first sequence does not matter. 

We conclude that the resulting diagram will look as the original one with the 
two components exchanged provided that the number of crossings is even in each 
box except perhaps the first one. Indeed, after moving all crossings from one
side to the other, it is enough to rotate the link of $\pi$ about a vertical
axis contained in the plane of the diagram to go back to the original diagram,
but now with the two components exchanged.

To reach the desired conclusion, we need to understand whether it is possible
to assume that the number of crossings is always even, except perhaps in the
first box at the top.

\begin{figure}[ht]
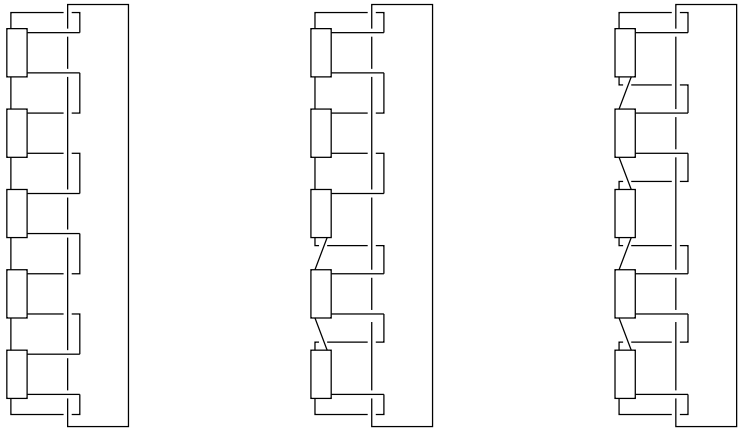

 \begin{center}



\ifx\XFigwidth\undefined\dimen1=0pt\else\dimen1\XFigwidth\fi
\divide\dimen1 by 4344
\ifx\XFigheight\undefined\dimen3=0pt\else\dimen3\XFigheight\fi
\divide\dimen3 by 2543
\ifdim\dimen1=0pt\ifdim\dimen3=0pt\dimen1=4143sp\dimen3\dimen1
  \else\dimen1\dimen3\fi\else\ifdim\dimen3=0pt\dimen3\dimen1\fi\fi
\tikzpicture[x=+\dimen1, y=+\dimen3]
{\ifx\XFigu\undefined\catcode`\@11
\def\temp{\alloc@1\dimen\dimendef\insc@unt}\temp\XFigu\catcode`\@12\fi}
\XFigu4143sp
\ifdim\XFigu<0pt\XFigu-\XFigu\fi
\clip(168,-2666) rectangle (4512,-123);
\tikzset{inner sep=+0pt, outer sep=+0pt}
\pgfsetlinewidth{+7.5\XFigu}
\draw (180,-279) rectangle (300,-567);
\draw (180,-759) rectangle (300,-1047);
\draw (180,-1718) rectangle (300,-2006);
\draw (180,-1239) rectangle (300,-1526);
\draw (180,-2198) rectangle (300,-2486);
\draw (204,-567)--(204,-759);
\draw (204,-1047)--(204,-1239);
\draw (204,-1526)--(204,-1718);
\draw (204,-2006)--(204,-2198);
\draw (300,-303)--(612,-303);
\draw (300,-543)--(612,-543);
\draw (300,-1263)--(612,-1263);
\draw (300,-1502)--(612,-1502);
\draw (300,-2222)--(612,-2222);
\draw (300,-2462)--(612,-2462);
\draw (300,-783)--(516,-783);
\draw (300,-1023)--(516,-1023);
\draw (300,-1742)--(516,-1742);
\draw (300,-1982)--(516,-1982);
\draw (204,-279)--(204,-183)--(516,-183);
\draw (204,-2486)--(204,-2582)--(516,-2582);
\draw (540,-327)--(540,-519);
\draw (540,-1287)--(540,-1478);
\draw (540,-2246)--(540,-2438);
\draw (540,-567)--(540,-1239);
\draw (540,-1526)--(540,-2198);
\draw (564,-183)--(612,-183)--(612,-303);
\draw (612,-2462)--(612,-2582)--(564,-2582);
\draw (564,-1982)--(612,-1982)--(612,-2222);
\draw (564,-1023)--(612,-1023)--(612,-1263);
\draw (612,-543)--(612,-783)--(564,-783);
\draw (612,-1502)--(612,-1742)--(564,-1742);
\draw (540,-279)--(540,-135)--(900,-135)--(900,-2654)--(540,-2654)--(540,-2486);
\draw (1980,-279) rectangle (2100,-567);
\draw (1980,-759) rectangle (2100,-1047);
\draw (1980,-1239) rectangle (2100,-1526);
\draw (1980,-2198) rectangle (2100,-2486);
\draw (2004,-567)--(2004,-759);
\draw (2004,-1047)--(2004,-1239);
\draw (2100,-303)--(2412,-303);
\draw (2100,-543)--(2412,-543);
\draw (2100,-1263)--(2412,-1263);
\draw (2100,-2462)--(2412,-2462);
\draw (2100,-783)--(2316,-783);
\draw (2100,-1023)--(2316,-1023);
\draw (2004,-279)--(2004,-183)--(2316,-183);
\draw (2004,-2486)--(2004,-2582)--(2316,-2582);
\draw (2340,-327)--(2340,-519);
\draw (2340,-1287)--(2340,-1718);
\draw (2340,-2030)--(2340,-2438);
\draw (2340,-567)--(2340,-1239);
\draw (2364,-183)--(2412,-183)--(2412,-303);
\draw (2412,-2462)--(2412,-2582)--(2364,-2582);
\draw (2364,-1023)--(2412,-1023)--(2412,-1263);
\draw (2412,-543)--(2412,-783)--(2364,-783);
\draw (2340,-279)--(2340,-135)--(2700,-135)--(2700,-2654)--(2340,-2654)--(2340,-2486);
\draw (1980,-1718) rectangle (2100,-2006);
\pgfsetdash{}{+0pt}
\draw (2076,-1526)--(2004,-1718);
\draw (2076,-2198)--(2004,-2006);
\pgfsetdash{}{+0pt}
\draw (2100,-1742)--(2412,-1742);
\draw (2100,-1982)--(2412,-1982);
\pgfsetdash{}{+0pt}
\draw (2004,-1526)--(2004,-1574)--(2028,-1574);
\draw (2004,-2198)--(2004,-2150)--(2028,-2150);
\pgfsetdash{}{+0pt}
\draw (2364,-1574)--(2412,-1574)--(2412,-1742);
\draw (2412,-1982)--(2412,-2150)--(2364,-2150);
\pgfsetdash{}{+0pt}
\draw (2076,-1574)--(2316,-1574);
\draw (2076,-2150)--(2316,-2150);
\draw (2340,-1766)--(2340,-1958);
\pgfsetdash{}{+0pt}
\draw (3780,-279) rectangle (3900,-567);
\draw (3780,-1239) rectangle (3900,-1526);
\draw (3780,-2198) rectangle (3900,-2486);
\draw (3900,-303)--(4212,-303);
\draw (3900,-2462)--(4212,-2462);
\draw (3804,-279)--(3804,-183)--(4116,-183);
\draw (3804,-2486)--(3804,-2582)--(4116,-2582);
\draw (4140,-327)--(4140,-759);
\draw (4140,-1047)--(4140,-1718);
\draw (4140,-2030)--(4140,-2438);
\draw (4164,-183)--(4212,-183)--(4212,-303);
\draw (4212,-2462)--(4212,-2582)--(4164,-2582);
\draw (4140,-279)--(4140,-135)--(4500,-135)--(4500,-2654)--(4140,-2654)--(4140,-2486);
\draw (3780,-759) rectangle (3900,-1047);
\pgfsetdash{}{+0pt}
\draw (3876,-567)--(3804,-759);
\draw (3876,-1239)--(3804,-1047);
\pgfsetdash{}{+0pt}
\draw (3900,-783)--(4212,-783);
\draw (3900,-1023)--(4212,-1023);
\pgfsetdash{}{+0pt}
\draw (3804,-567)--(3804,-615)--(3828,-615);
\draw (3804,-1239)--(3804,-1191)--(3828,-1191);
\pgfsetdash{}{+0pt}
\draw (4164,-615)--(4212,-615)--(4212,-783);
\draw (4212,-1023)--(4212,-1191)--(4164,-1191);
\pgfsetdash{}{+0pt}
\draw (3876,-615)--(4116,-615);
\draw (3876,-1191)--(4116,-1191);
\draw (4140,-807)--(4140,-999);
\pgfsetdash{}{+0pt}
\draw (3780,-1718) rectangle (3900,-2006);
\pgfsetdash{}{+0pt}
\draw (3876,-1526)--(3804,-1718);
\draw (3876,-2198)--(3804,-2006);
\pgfsetdash{}{+0pt}
\draw (3900,-1742)--(4212,-1742);
\draw (3900,-1982)--(4212,-1982);
\pgfsetdash{}{+0pt}
\draw (3804,-1526)--(3804,-1574)--(3828,-1574);
\draw (3804,-2198)--(3804,-2150)--(3828,-2150);
\pgfsetdash{}{+0pt}
\draw (4164,-1574)--(4212,-1574)--(4212,-1742);
\draw (4212,-1982)--(4212,-2150)--(4164,-2150);
\pgfsetdash{}{+0pt}
\draw (3876,-1574)--(4116,-1574);
\draw (3876,-2150)--(4116,-2150);
\draw (4140,-1766)--(4140,-1958);
\endtikzpicture%

\end{center}
  \caption{How to change the parity of the number of crossings contained in a
box.}
 \label{fig:crosschange}
\end{figure}

Figure~\ref{fig:crosschange} shows how to change the parity of the number of
crossings in the sequences, so that all boxes, except perhaps the first one,
contain an even number of crossings. Consider the diagram shown on the
right-hand side of Figure~\ref{fig:crosschange} and assume that we want to
change the parity of the number of crossings in the last box at the bottom. To
achieve that, take the next box and make it go around the arc of the second 
component by passing it either first under and then over (as in the situation
we are considering) or first over and then under (if we were to move, say, the
central box). The result is shown in the central diagram of 
Figure~\ref{fig:crosschange} where we see that the number of crossings of the
adjacent boxes has changed by one. One might now want to repeat this operation
whenever needed to change the parity of the number of crossings in the boxes 
starting from the bottom and going up. Unfortunately, this is not possible 
right away. Indeed, the central box of the central diagram is not positioned 
as the boxes of the diagram on the left-hand side, so it cannot be moved as 
explained. In fact, one remarks that in the initial diagram the boxes of the 
first component are positioned in such a way that they alternate between boxes 
lying ``above" with respect to the second component and boxes lying ``below". 
To restore the same structure as in the original diagram for all the boxes 
above those already successfully modified, it is necessary to move not only one 
box, but all those preceeding it that are an even number of boxes away from it. 
In this specific example, when we move the second box from the bottom, we have 
to move the second box from the top as well. The result is the diagram on the
right-hand side of Figure~\ref{fig:crosschange}. On the other hand, if we had
to move the central box, we would need to move the first box at the top as
well. The effect of this will be that the positions of all the boxes above the 
one we are dealing with will be exchanged: those which used to lie above with
respect to the second component will end up below, and vice versa. This way, we 
also only change the parity of the number of crossings in the box we are 
considering and possibly that of the first box at the top. 

Note that this process can change the type of crossings between the two
components, but regardless of what these are they will be preserved when 
performing an even number of flypes anyway.

This achieves the proof of the  proposition.
\end{proof}


\section{Proof of Theorem~\ref{t:main}}\label{s:proof}

We are now in a position to prove Thoerem~\ref{t:main}. 

Let $K$ be a prime alternating knot and let $n>2$ be an integer. By \cite{Me}, 
we know  that $K$ is either a torus knot or a hyperbolic knot. More precisely, 
if $K$ is a torus knot, then it must be a $(2,2m+1)$-torus knot according to
Lackenby's characterisation of tunnel number-one alternating knots in \cite{L}.

It follows from Thurson's orbifold theorem (see \cite{BLP,BMP,BPo,CHK} for a 
proof) that if $K'$ is a $n$-twin of $K$ then $K'$ is the same type of knot as 
$K$, that is a torus knot (if so is $K$), or a hyperbolic knot (if so is $K$). 
Here the fact that $n$ is at least $3$ is crucial.

The following result deals with the case of arbitrary torus knots.

\begin{Proposition}\label{p:torus}
Let $n\ge 2$. Two torus knots cannot be $n$-twins.
\end{Proposition}

We postpone the proof of the proposition at the end of this section. At this
point we just remark that, if on the one hand the proposition suffices to 
ensure that a torus knot cannot have $n$-twins for $n>2$ because of Thurston's
orbifold theorem, on the other, there are torus knots that have $2$-twins that 
are Montesinos knots. Alternating torus knots, however, do not have twins at 
all since they are $2$-bridge knots. Indeed, Hodgson and Rubinstein proved in
\cite{HR} that $2$-bridge knots are detrmined by they $2$-fold cyclic branched 
covers. 

Let $K$ be hyperbolic and assume by contradiction that $K$ is not determined
by its $n$-fold cyclic branched cover. It follows from \cite[Theorem 3]{Z}
that $K$ admits a period $\psi$ of order $n$ such that the knot 
$K/\langle \psi \rangle$ is trivial and no orientation preserving
homeomorphism of the $3$-sphere exchanges the two components of the link 
$(K\cup Fix(\psi))/\langle \psi \rangle$. Note that \cite[Theorem 3]{Z} is 
stated only for $n$ not a power of $2$, but in fact it holds for any $n>2$ and
the proof uses basically the same argument (see, for instance, 
\cite[Chapter 4]{P1}). According to the orbifold theorem, there is a single 
case where \cite[Theorem 3]{Z} does not apply, that is when $K=4_1$ is the 
figure-eight knot and $n=3$. It is however well-known as a consequence of
Dunbar's classification of geometric orbifolds in \cite{D} that the knot 
$K=4_1$ has no $3$-twins.

Since $K$ is alternating and $\psi$ is a period of order $>2$, it was proved 
in \cite{CQ} that $\psi$ is visible on a minimal alternating diagram for $K$. 
As a consequence Proposition~\ref{p:quodiag} applies as well as 
Proposition~\ref{p:symdiag}, providing the desired contradiction.

This ends the proof of the theorem. We can now pass to the proof the result
about torus knots.

\begin{proof}[Proof of Proposition~\ref{p:torus}]
In what follows we will use the description of the $n$-fold cyclic coverings of
torus knots provided by N\'u\~nez and Ram\'{\i}rez-Losada in 
\cite[Theorem 1]{NRL}. We will restate their result in a way that is more
convenient for us. Let $a_1\ge 2$ and $a_2\ge 2$ be two coprime integers, and 
$n\ge 2$. Let $d_i$, $i=1,2$, be the greatest common divisor of $a_i$ and $n$, 
so that $d=d_1d_2$ is the greatest common divisor of $a_1a_2$ and $n$. 
The $n$-fold cyclic branched covering of the $(a_1,a_2)$-torus knot is a
Seifert fibred space with orientable base and exceptional fibres of at most
three distinct orders which are moreover pairwise coprime. According to the 
different properties of $d_1$, $d_2$, and $d$, the Seifert invariants of the 
space satisfy the following conditions.
\begin{enumerate}

\item If $d=d_1=d_2=1$, then the base of the fibration is the $2$-sphere and
there is one exceptional fibre of order $a_2$, one of order $a_1$, and one of 
order $n$.

\item If one between $d_1$ and $d_2$ is equal $1$ but the other is $>1$, we
can assume without loss of generality that $d_1>1$ and $d_2=1$ for here $a_1$ 
and $a_2$ play symmetric roles (we are not assuming $a_1<a_2$, for instance,
and of course the  $(a_1,a_2)$-torus knot is equivalent to the 
$(a_2,a_1)$-torus knot).

\begin{enumerate}
\item If $d=d_1<a_1$ and $d<n$, then the base of the fibration is the 
$2$-sphere and there are $d$ exceptional fibres of order $a_2$, one of order 
$a_1/d$, and one of order $n/d$.

\item If $d=d_1<a_1$ and $d=n$, then the base of the fibration is the 
$2$-sphere and there are $n$ exceptional fibres of order $a_2$, and one of 
order $a_1/d$.

\item If $d=d_1=a_1$ and $d<n$, then the base of the fibration is the 
$2$-sphere and there are $d$ exceptional fibres of order $a_2$, and one of 
order $n/d$.

\item If $d=d_1=a_1=n$, then the base of the fibration is the
$2$-sphere and there are $n$ exceptional fibres of order $a_2$.

\end{enumerate}

\item We can now assume that both $d_1$ and $d_2$ are $>1$.

\begin{enumerate}

\item If $d_1<a_1$, $d_2<a_2$, and $d<n$, then the base of the fibration is a
surface of genus $g=(d_1-1)(d_2-1)/2>0$ and there are $d_1$ exceptional fibres 
of order $a_2/d_2$, $d_2$ of order $a_1/d_1$, and one of order $n/d$.

\item If $d_1<a_1$, $d_2<a_2$, and $d=n$, then the base of the fibration is a
surface of genus $g=(n+1-d_1-1-d_2)/2>0$ and there are $d_1$ exceptional fibres 
of order $a_2/d_2$ and $d_2$ of order $a_1/d_1$.

\item If $d_1=a_1$, $d_2<a_2$, and $d<n$, then the base of the fibration is a
surface of genus $g=(d_1-1)(d_2-1)/2>0$ and there are $d_1$ exceptional fibres 
of order $a_2/d_2$ and one of order $n/d$.

\item If $d_1=a_1$, $d_2<a_2$, and $d=n$, then the base of the fibration is a
surface of genus $g=(n+1-d_1-1-d_2)/2>0$ and there are $d_1$ exceptional fibres
of order $a_2/d_2$.

\item If $d_1=a_1$, $d_2=a_2$, and $d<n$, then the base of the fibration is a
surface of genus $g=(d_1-1)(d_2-1)/2>0$ and there is one exceptional fibre of 
order $n/d$.

\item If $d_1=a_1$, $d_2=a_2$, and $d=n$, then the base of the fibration is a
surface of genus $g=(n+1-d_1-1-d_2)/2>0$ and there are no exceptional fibres.

\end{enumerate}

\end{enumerate}

Let $n\ge 2$ be fixed and let $M$ be the $n$-fold cyclic branched covering of
some torus knot. We claim that there is a unique torus knot $K$ such that $M$ 
is homeomorphic to $M(K,n)$. Let us assume that $M$ admits more than one
Seifert fibration. Since $M$ is a closed manifold, $M$ is either a lens space
of a prism manifold (see \cite[VI.16]{J}). If $M$ is a lens space, the base of 
a fibration is the $2$-sphere if it is orientable, and the fibration has at 
most two exceptional fibres, so the only case where lens spaces appear in the 
above list is case 2d for $n=2$. This means that $M$ must be a $2$-fold 
branched cover and $K$ a $2$-bridge knot. Since it was shown by Hodgson and 
Rubinstein \cite{HR} that $2$-bridge knots are determined by their $2$-fold 
branched covers, we can ignore this case. If $M$ is a prism manifold it has a 
single fibration with orientable base as those of our list, for the other 
fibration has the projective plane as base. We can thus assume that $M$ is not 
a lens space and has a unique fibration among those given in the above list. 

Assume first that the base of the fibration is the sphere. If
there are three different types of exceptional fibres then we are either in
situation 1 or in situation 2a. If there are exactly three exceptional fibres 
we are in situation 1 and if there are strictly more than three we are in 
situation 2a. In both cases, since $n$ is fixed, it is possible to recover the 
invariants of the torus knot from the orders of the exceptional fibres. If 
there are two different types of exceptional fibres we are in cases 2b or 2c. 
These two cases can be distinguished by the fact that in 2b the orders of the 
exceptional fibres are both coprime with $n$, while in 2c one of them is not. 
Once again, in both cases, $a_1$ and $a_2$ can be retrieved from the Seifert 
invariants. Finally, if there is only one type of exceptional fibre, we are in 
case 2d and again it is possible to reconstruct $a_1$ and $a_2$ from the number 
of exceptional fibres and their order.

Assume now that the base of the fibration is an oriantable surface of genus
$>0$. If there are three types of exceptional fibres we are in case 3a, if
there are two in cases 3b or 3c, if there is one in cases 3d or 3e, and if 
there is none we are in case 3f. In case 3c there is just one exceptional fibre 
of one of the two types while in case 3d there is more than one exceptional 
fibre of both types. A similar argument allows to distinguish cases 3d and 3e. 
What may be not completely obvious in this situation is how to recover 
$a_1=d_1$ and $a_2=d_2$ in cases 3e and 3f. Let us assume we are in case 3e. 
Since we know $n$, the order $s$ of the exceptional fibre allows to compute 
$d_1d_2=d=n/s$, while the genus allows to compute 
$d_1+d_2=d_1d_2+1-2g=n/s+1-2g$ which is enough to obtain the values of 
$a_1=d_1$ and $a_2=d_2$. Case 3f is similar, keeping in mind that now 
$d_1d_2=d=n$. This completes the proof of the proposition.
\end{proof}


\paragraph{Acknowledgements} 
The contents of this work were inspired by J. Greene's paper on double branched
covers of prime alternating knots. They heavily rely on the existence of a
periodic minimal diagram for a prime alternating knot. In her quest for a
proof of this fact, the author pestered several colleagues: M. Boileau, A.
Costa, J. Greene, C. V. Quach Hongler, M. Thistlethwaite, A. Tsvietkova, 
J. Weeks, and possibly others she may have forgotten along the way. She wishes 
to express them all her deep gratitude for their patience and comments. She is
also indebted to L. Watson for drawing her attention to the peculiarities of 
alternating diagrams of the trivial knot and B. Owen for pointing out an
imprecision in the first version of the paper.

\textsc{Aix-Marseille Univ, CNRS, Centrale Marseille, I2M, UMR 7373,
13453 Marseille, France}

{luisa.paoluzzi@univ-amu.fr}

\end{document}